\documentclass[3p, 12pt]{elsarticle}

\pdfoutput=1

\usepackage{amssymb}
\usepackage{amsthm}
\usepackage{appendix}

\usepackage{float}
\usepackage{amsfonts,amsmath,latexsym}
\usepackage[active]{srcltx}
\usepackage{enumitem}
\usepackage{amsfonts,amsmath,latexsym}
\usepackage{mathtools}

\usepackage{etoolbox} 
\AtBeginEnvironment{pmatrix}{\everymath{\displaystyle}}
\AtBeginEnvironment{bmatrix}{\everymath{\displaystyle}}

\graphicspath{{Pictures/}}

\newtheorem{theorem}{Theorem}[section]
\newtheorem{remark}[theorem]{Remark}

\newtheorem{proposition}[theorem]{Proposition}
\newproof{pf}{Proof}

\newcommand{\RR}{{\if mm {\rm I}\mkern -3mu{\rm R}\else \leavevmode
\hbox{I}\kern -.17em\hbox{R} \fi}}

\newcommand{\bu}{\mbox{\boldmath{$u$}}}

\newcommand{\bcero}{\mbox{\boldmath{$0$}}}
\newcommand{\bvarphi}{\mbox{\boldmath{$\varphi$}}}
\newcommand{\bV}{\mbox{\boldmath{$V$}}}
\newcommand{\bx}{\mbox{\boldmath{$\bx$}}}
\newcommand{\by}{\mbox{\boldmath{$\by$}}}

\newcommand{\bs}{\mbox{\boldmath{$s$}}}

\newcommand{\var}{\varepsilon}

\renewcommand{\by}{\mbox{\boldmath{$y$}}}
\renewcommand{\bx}{\mbox{\boldmath{$x$}}}
\newcommand{\be}{\mbox{\boldmath{$e$}}}
\newcommand{\bn}{\mbox{\boldmath{$n$}}}

\newcommand{\bz}{\mathbf{z}}

\renewcommand{\d}{\partial}

\newcommand{\pfrac}[2]{\frac{\d #1}{\d #2}}

\newcommand{\dpc}{\frac{\d p^0}{\d s_1}}
\newcommand{\dpu}{\frac{\d p^1}{\d s_1}}
\newcommand{\dpcd}{\frac{\d^2 p^0}{\d s_1^2}}
\newcommand{\dpud}{\frac{\d^2 p^1}{\d s_1^2}}
\newcommand{\dR}{\frac{\d R}{\d s_1}}

\newcommand{\rhonu}{\rho_0\nu}

\newcommand{\bbv}{\mathbf{v}}
\newcommand{\bc}{\mathbf{c}}

\newcommand{\bW}{\mathbf{W}}

\newcommand{\bphi}{\mbox{\boldmath{$\varphi$}}}

\newcommand{\bchi}{\mbox{\boldmath{$\chi$}}}

\journal{arXiv.org}

\begin{document}

\begin{frontmatter}

\title{Asymptotic Analysis of a Viscous Fluid in a Curved Pipe with Elastic Walls}

\author[compostela]{G.~Casti\~neira}
\ead{gonzalo.castineira@usc.es}
\author[corunha]{J. M. Rodr\'{\i}guez}
\ead{jose.rodriguez.seijo@udc.es}

\address[compostela]{Departamento de Matem\'atica Aplicada,
Univ. de Santiago de Compostela, Spain}
\address[corunha]{Departamento de M\'etodos Matem\'aticos
e Representaci\'on, Univ. da Coru\~na, Spain}

\begin{abstract}
This communication is devoted to the presentation of our recent results regarding
the asymptotic analysis of a viscous flow in a tube with elastic walls. This study can be applied, 
for example, to the blood flow in an artery. With this aim, we consider the dynamic problem of the incompressible flow of a viscous fluid through a curved pipe with a smooth central curve. Our analysis leads to obtain an one dimensional model via singular perturbation of the Navier-Stokes system as $\varepsilon$, a non dimensional parameter related to the radius of cross-section of the tube, tends to zero. We allow the radius depend on tangential direction and time, so a coupling with an elastic or viscoelastic law on the wall of the pipe is possible.

To perform the asymptotic analysis, we do a change of variable to a reference domain where we assume the existence of asymptotic expansions on $\varepsilon$ for both velocity and pressure which, upon substitution on Navier-Stokes equations, leads to the characterization 
of various terms of the expansion. This allows us to obtain an approximation of the solution of 
the Navier-Stokes equations.
\end{abstract}


\end{frontmatter}

\section{Introduction}
Last decades, applied mathematics have been involved in some new fields where they had not been applied before. One of these fields is biomedicine, from which new methods to improve the diagnosis and treatment of different diseases are demanded. In particular, in the case of cardiovascular problems, modeling the blood flow in veins and arteries is a difficult problem.

A large number of articles have studied the flow of a viscous fluid through a pipe. For example, in 
\cite{Hydon,Paloka_helical,Smith} the flow behavior inside the pipe is related with the curvature and torsion of its middle line. In \cite{Hydon} the main term of the asymptotic expansion of the solution is compared with a Poiseuille flow inside a pipe with rigid walls. 
In \cite{Panasenko}, the same
problem but with visco-elastic walls is considered, leading to a 
fluid-structure problem.
In \cite{Lyne} the secondary flow is studied, the boundary layer in \cite{Riley}, 
both depending on values of Dean number. More recently, the non-steady case in tube structures, has been considered in \cite{Panasenko_trans,Panasenko_trans2}, where estimates of the error between exact solution and the asymptotic approximation are proved.

 There are also articles where the flow in blood vessels is modeled. An one dimensional model is presented in \cite{Quarteroni}, where 
  clinical procedures where this model can be useful are highlighted. Another model for blood flow in arteries is developed in \cite{Pedley}, relating blood pulse and flow patterns, and remarking how this kind of models can help with the design of treatments for particular diseases.

In this article, we shall follow the spirits of \cite{Paloka}, where asymptotic analysis is used to find a model for a steady flow through a curved pipe with rigid walls. We shall consider, instead, an unsteady flow and elastic walls. The structure of this article is the following: in section \ref{sec-2} we shall describe the problem in a reference domain, in section \ref{sec-3} we shall suppose the existence of an asymptotic expansion of the solution and we shall identify the first terms of this expansion, in section \ref{sec-4} we shall show some examples of the tangential and transversal velocity, and finally, we shall present some conclusions in section \ref{sec-5}.

\section{Setting the problem in a reference domain} \label{sec-2}

Let us suppose that central curve of the pipe is parametrized by $\mathbf{c}(s)$, 
where $s \in [0, L]$ is the arc-length parameter, and the interior points of the 
pipe are given by 
$$(x,y,z) = \mathbf{c}(s) + 
\varepsilon\, r\, R(t,s)\left [ (\cos \theta) \textbf{N}(s) \\ + (\sin \theta)  \textbf{B}(s) \right ],$$ 
where $r \in [0,1]$, $\theta \in [0, 2\pi]$, $\{ \textbf{T} = \mathbf{c}', \textbf{N}, \textbf{B} \}$ is 
the Frenet-Serret frame of $\mathbf{c}$, and 
$\varepsilon R(t,s)$ is the radius of the cross-section of 
the pipe at point $\mathbf{c}(s)$ and time $t$. The non dimensional parameter $\varepsilon$ 
represents the different scale of magnitude between the pipe diameter and its length, so we shall assume that $\varepsilon << 1$.

Let us introduce the following notation, $s_1:=s, s_2:=\theta , s_3:=r $ for the variables, and $\{  \textbf{v}_1:= \textbf{T}, \textbf{v}_2:=\textbf{N},\textbf{ v}_3:=\textbf{B} \}$, for the Frenet-Serret frame of $\mathbf{c}$. This new notation will allow us to use Einstein summation convention in what follows.

Let be the subsets of $\mathbb{R}^3$ defined by $\Omega^\varepsilon=[0,L]\times[0,2\pi]\times[0,\varepsilon]$ and $\Omega=[0,L]\times[0,2\pi]\times[0,1]$. We define the maps 
$\phi_1^{\varepsilon}:\Omega \rightarrow \Omega^{\varepsilon}$, $\phi_2^{\varepsilon}:\Omega^{\varepsilon} \rightarrow \hat{\Omega}_t^\varepsilon$, 
where $\phi_1^{\varepsilon}$ and $\phi_2^{\varepsilon}$ are 
given by the expressions, 
\begin{equation} \label{eq-phi-1-2}
\left.\begin{array}{l}
\phi_1^{\varepsilon}(s_1, s_2, s_3)=(s_1, s_2, \varepsilon s_3)=:(s_1^{\varepsilon}, s_2^{\varepsilon}, s_3^{\varepsilon}),\\[.8em]
\phi_2^{\varepsilon}(s_1^{\varepsilon}, s_2^{\varepsilon}, s_3^{\varepsilon})= \mathbf{c}(s_1^{\varepsilon}) +
 s_3^{\varepsilon} R(t,s_1^{\varepsilon})[ (\cos s_2^{\varepsilon}) \textbf{v}_2(s_1^{\varepsilon}) + (\sin s_2^{\varepsilon}) \textbf{v}_3(s_1^{\varepsilon})  ],
\end{array}\right.
\end{equation}
and $\hat{\Omega}_t^\varepsilon = \phi_2^{\varepsilon} \left ( \phi_1^{\varepsilon} 
\left ( \Omega \right ) \right )$ represents the interior points of the pipe.

We can then introduce the change of variable from the reference domain $\Omega$, $$ \displaystyle \phi^\varepsilon=\left( \phi_2^{\varepsilon}\circ \phi_1^{\varepsilon} \right): \Omega \rightarrow \hat{\Omega}_t^\varepsilon,$$
\begin{equation}\label{aplicacion}
\left.\begin{array}{l}
 \phi^\varepsilon(s_1, s_2, s_3)= \mathbf{c}(s_1) + \varepsilon s_3 R(t,s_1)[ (\cos s_2) \textbf{v}_2(s_1) 
 + (\sin s_2)\textbf{v}_3(s_1)]=:(x_1^{\varepsilon}, x_2^{\varepsilon}, x_3^{\varepsilon}).
\end{array}\right.
\end{equation}

Let us consider the incompressible Navier-Stokes equations in the domain 
$\hat{\Omega}_t^\varepsilon$ given by,
\begin{eqnarray}
&&\frac{\partial \mathbf{u}^\varepsilon}{\partial t} +  (\nabla \mathbf{u}^\varepsilon) \mathbf{u}^\varepsilon = \frac{1}{\rho_0} \mathrm{div}\, \mathbf{T}^{\varepsilon} + \mathbf{b}_0^{\varepsilon},  \\
&&\mathrm{div}\, \mathbf{u}^\varepsilon = 0,
\end{eqnarray}
where $\mathbf{u}^\varepsilon$ stands for the velocity field, 
$\mathbf{b}_0^{\varepsilon}$ is the density of body forces and $\mathbf{T}^{\varepsilon}$ is the stress tensor given by
 $$\mathbf{T}^{\varepsilon}= -p^\varepsilon\mathbf{I} + 2\mu \mathbf{\Sigma}^\varepsilon,$$
where $p^\varepsilon$ is the pressure field, $\mu$ the dynamic viscosity and 
$\mathbf{\Sigma^\varepsilon} = \frac{1}{2} \left( \nabla \textbf{u}^\varepsilon + (\nabla \textbf{u}^\varepsilon)^T\right)$. Let $\nu= \mu/\rho_0$ be the kinematic viscosity, 
so we can write these equations,
\begin{eqnarray} \label{NS_eps1}
&&\frac{\partial \mathbf{u}^\varepsilon}{\partial t} +  (\nabla \mathbf{u}^\varepsilon) \mathbf{u}^\varepsilon + \frac{1}{\rho_0}\nabla p^{\varepsilon} -\nu \Delta \mathbf{u}^{\varepsilon} 
= \mathbf{b}_0^{\varepsilon},  \\\label{NS_eps2}
&&\mathrm{div}\, \mathbf{u}^\varepsilon = 0.
\end{eqnarray}

We shall consider continuity between the fluid and the wall of the pipe displacements. Let us suppose that only radial displacements of the wall are allowed. Then the boundary condition at the interface of the fluid and the wall of the pipe can be expressed as
\begin{equation}
\label{bc-1}
\mathbf{u}^{\varepsilon}=\left(\varepsilon \frac{\partial R}{\partial t} \right)\mathbf{n}^\varepsilon 
\ \textrm{at}\  s_3^\varepsilon=\varepsilon,
\end{equation}
where $\mathbf{n}^\varepsilon$ is the outward unitary normal at $s_3^\varepsilon=\varepsilon$.

 %estudiar la inversa de la aplicacion de referencia
 The next step is to write the equations of the problem in the reference domain $\Omega$. Taking 
 into account the change of variable (\ref{aplicacion}), we can associate to each vector field 
 $\mathbf{w}^{\varepsilon}$ in $\hat{\Omega}_t^\varepsilon$, a new vector field 
 $\mathbf{w}(\varepsilon)$ defined in $\Omega$, as follows
 \begin{align} \label{trans_dominioreferencia}
w_i^{\varepsilon}=\mathbf{w}^{\varepsilon}\cdot \mathbf{e}_i= (w_k^{\varepsilon} \mathbf{e}_k)\cdot \mathbf{e}_i= (w_k(\varepsilon) \mathbf{v}_k)\cdot \mathbf{e}_i=:w_k(\varepsilon) v_{ki},
 \end{align}
where $\{ \mathbf{e}_1, \mathbf{e}_2, \mathbf{e}_3 \}$ is an orthonormal basis, we are using 
the Einstein summation convention (where latin indices indicate sum from 1 to 3), and 
we denote $v_{ki}:=\mathbf{v}_k\cdot \mathbf{e}_i$.

%estudio de los terminos ds/dx

\subsection{Computing the Jacobian of the inverse mapping of the change of variable}

As first step, we shall need to study the inverse mapping of the change of variable (\ref{aplicacion}), in particular, of its Jacobian, which terms will be needed to write Navier-Stokes equations in the reference domain. Let us consider the mapping:
\begin{align}
& \tilde{\phi}^\varepsilon :[0,T]\times\Omega\longrightarrow[0,T]\times  \hat{\Omega}_t^\varepsilon\\
& \tilde{\phi}^\varepsilon(t,s_1,s_2,s_3):=(t^\varepsilon,x^\varepsilon)=(t^\varepsilon, \phi^\varepsilon(s_1,s_2,s_3)),
\end{align}
hence, the associated Jacobian denoted by $J_\phi$ is
\[
J_\phi=\begin{pmatrix}
 \pfrac{t^\var}{t}   & \pfrac{t^\var}{s_1}  & \pfrac{t^\var}{s_2}   & \pfrac{t^\var}{s_3} \\
 &&& \\
 \pfrac{x_1^\var}{t} & \pfrac{x_1^\var}{s_1}& \pfrac{x_1^\var}{s_2} & \pfrac{x_1^\var}{s_3} \\
 &&& \\
 \pfrac{x_2^\var}{t} & \pfrac{x_2^\var}{s_1}& \pfrac{x_2^\var}{s_2} & \pfrac{x_2^\var}{s_3} \\
 &&& \\
 \pfrac{x_3^\var}{t} & \pfrac{x_3^\var}{s_1}& \pfrac{x_3^\var}{s_2} & \pfrac{x_3^\var}{s_3}
\end{pmatrix} = \begin{pmatrix}
 \pfrac{t^\var}{t}   & \pfrac{t^\var}{s_1}  & \pfrac{t^\var}{s_2}   & \pfrac{t^\var}{s_3} \\ 
 &&& \\
 %\vdots       & \ddots &  &  \\
 \displaystyle\pfrac{\bx^\var}{t} & &\displaystyle \nabla_{\mathbf{s}}\bx^\var &  
 %\\ \vdots& & & \ddots
\end{pmatrix},
\]
where $\mathbf{s}=(s_1,s_2,s_3)$. Since
\begin{align*} 
t^\var&=t, \\
\bx^\var&= \mathbf{c}(s_1)+ \var s_3 R(t,s_1)((\cos s_2) \bbv_2(s_1) +( \sin s_2) \bbv_3(s_1) ), 
\end{align*}
then,
\begin{align*}
\pfrac{t^\var}{t}&=1,\\
\pfrac{t^\var}{s_i}&=0, \\
\pfrac{\bx^\var}{t}&=\var s_3 \pfrac{R}{t} ( (\cos s_2) \bbv_2(s_1) + (\sin s_2) \bbv_3(s_1)), \\
\pfrac{\bx^\var}{s_1}&=\mathbf{c}'(s_1) + \var s_3 \pfrac{R}{s_1}((\cos s_2) \bbv_2(s_1) + (\sin s_2) \bbv_3(s_1)) + \var s_3 R((\cos s_2) \bbv_2'(s_1)+ (\sin s_2) \bbv_3'(s_1)), \\
\pfrac{\bx^\var}{s_2}&= \var s_3 R (-(\sin s_2) \bbv_2(s_1)+( \cos s_2) \bbv_3 (s_1)),\\
\pfrac{\bx^\var}{s_3}&= \var R ((\cos s_2) \bbv_2(s_1)+ (\sin s_2) \bbv_3 (s_1)).
\end{align*}

The inverse mapping $(\tilde{\phi}^\varepsilon)^{-1} :[0,T]\times  \hat{\Omega}_t^\varepsilon \longrightarrow [0,T]\times\Omega$ is such that its Jacobian, denoted by $J_\phi^{-1}$, is
\[
J_\phi^{-1}=\begin{pmatrix}
 \pfrac{t}{t^\var}   & \pfrac{t}{x_1^\var}  & \pfrac{t}{x_2^\var}   & \pfrac{t}{x_3^\var} \\
 \pfrac{s_1}{t^\var} & \pfrac{s_1}{x_1^\var}& \pfrac{s_1}{x_2^\var} & \pfrac{s_1}{x_3^\var} \\
 \pfrac{s_2}{t^\var} & \pfrac{s_2}{x_1^\var}& \pfrac{s_2}{x_2^\var} & \pfrac{s_2}{x_3^\var} \\
 \pfrac{s_3}{t^\var} & \pfrac{s_3}{x_1^\var}& \pfrac{s_3}{x_2^\var} & \pfrac{s_3}{x_3^\var}
\end{pmatrix} = \begin{pmatrix} 
 1   & 0  & 0  & 0 \\
 &&& \\
 %          \vdots    &\ddots &  &  \\
 \displaystyle\pfrac{\mathbf{s}}{t^\var} & &\displaystyle \nabla_{\bx^\var}\mathbf{s} &  
%  \\ \vdots & & & \ddots
\end{pmatrix}.
\]

Therefore, since $J_\phi J_\phi^{-1}= \mathbf{I}$, we find the following relations,
\begin{align} \label{relacion_gradientes}
\nabla_{\bx^\var} \mathbf{s}&= (\nabla_{\mathbf{s}}\bx^\var)^{-1}, \\\label{relacion_der_temporal}
\pfrac{\bs}{t^\var}&= -(\nabla_{\bx^\var} \mathbf{s}) \pfrac{\bx^\var}{t}.
\end{align}

In order to compute  $\displaystyle \pfrac{s_i}{\bx^\var}$,  let us write 
$\displaystyle \pfrac{s_i}{ \bx^\var} = \alpha_i \bbv_1 + \beta_i \bbv_2 + \gamma_i \bbv_3$ . 

We know, by Frenet-Serret formulas, that the following equalities hold 
\begin{align} \label{Frenet}
\left\{\begin{aligned}[c]
\bbv_1'(s_1)&= \kappa(s_1)\bbv_2(s_1),\\
\bbv_2'(s_1)&=- \kappa(s_1)\bbv_1(s_1)+ \tau(s_1) \bbv_3(s_1),\\
\bbv_3'(s_1)&=-\tau(s_1) \bbv_2(s_1),
\end{aligned}\right.\qquad
\end{align}
where the functions $\kappa$ and $\tau$ denote the curvature and torsion of the middle line of the curved pipe. 
Now, by (\ref{relacion_gradientes}), we have   that
\begin{align} \label{producto_igual_delta}
\pfrac{s_i}{ \bx^\var}\cdot \pfrac{ \bx^\varepsilon}{s_j}=\delta_{ij},
\end{align}
where $\delta_{ij}$ is the Kronecker's delta.  For $i=1$ we find that
\begin{align} \nonumber
1&=\left(\alpha_1\bbv_1+ \beta_1\bbv_2+ \gamma_1\bbv _3\right) \cdot 
\bigg(\bbv_1 + \var s_3 \pfrac{R}{s_1} (\cos s_2 \bbv_2 + \sin s_2 \bbv_3) \\ \nonumber
&\quad  {} + \var s_3 R (\cos s_2 \bbv_2' + \sin s_2 \bbv_3')    \bigg)\\ \nonumber
&= \alpha_1 \left( 1 + \var s_3 R \left(\cos s_2(\bbv_2'\cdot\bbv_1) + \sin s_2 (\bbv_3'\cdot \bbv_1)\right)\right)  \\ \nonumber
&\quad {} + \beta_1 \var s_3 \left( \pfrac{R}{s_1} \cos s_2 + R\sin s_2 (\bbv_3'\cdot\bbv_2)\right) 
\\ \label{et3}
&\quad {} + \gamma_1 \var s_3 \left( \pfrac{R}{s_1}\sin s_2 + R\cos s_2( \bbv_2' \cdot\bbv_3)\right), 
\end{align}
since $\displaystyle \bbv_1=\bc'$, and  
\begin{align}\nonumber
0&=\left(\alpha_1\bbv_1+ \beta_1\bbv_2+ \gamma_1\bbv _3\right)\cdot\left(\var s_3 R(-\sin s_2 \bbv_2 + \cos s_2 \bbv_3) \right)
\\ \label{et1}
&=-\beta_1\var s_3R\sin s_2 + \gamma_1 \var s_3 R \cos s_2,
\\ \nonumber
0&=\left(\alpha_1\bbv_1+ \beta_1\bbv_2+ \gamma_1\bbv _3\right)\cdot\left(\var  R(\cos s_2 \bbv_2 + \sin s_2 \bbv_3) \right) 
\\ \label{et2}
&=\beta_1\var R\cos s_2 + \gamma_1 \var  R \sin s_2.
\end{align}

From (\ref{et1})--(\ref{et2}) is easy to check that $\beta_1=\gamma_1=0$. Hence, from (\ref{et3}) and (\ref{Frenet}), we obtain that
\begin{align}
\alpha_1=\frac{1}{1-\var \kappa(s_1) s_3 R(t,s_1)\cos s_2 }
\end{align}
Now, for $i=2$ in (\ref{producto_igual_delta}) we find, on one hand, that
\begin{align} \nonumber
0&= \alpha_2 \left( 1 + \var s_3 R \left(\cos s_2(\bbv_2'\cdot\bbv_1) + \sin s_2 (\bbv_3'\cdot \bbv_1)\right)\right) \\ \nonumber 
&\quad {} + \beta_2 \var s_3 \left( \pfrac{R}{s_1} \cos s_2 + R\sin s_2 (\bbv_3'\cdot\bbv_2)\right) \\ \label{et4}
&\quad {} + \gamma_2 \var s_3 \left( \pfrac{R}{s_1}\sin s_2 + R\cos s_2( \bbv_2' \cdot\bbv_3)\right), 
\end{align}
and, on the other hand, that
\begin{align}\nonumber
1&=-\beta_2\var s_3R\sin s_2 + \gamma_2 \var s_3 R \cos s_2
\\ \nonumber
0&=\beta_2\var R\cos s_2 + \gamma_2 \var  R \sin s_2.
\end{align}
where we deduce that,
\begin{align}
\beta_2= - \frac{\sin s_2}{\var s_3 R(t,s_1)}, \quad \gamma_2=\frac{\cos s_2}{\var s_3 R(t,s_1)}.  
\end{align}
 Therefore, from (\ref{Frenet}) and (\ref{et4}), we obtain that
 \begin{align}
 \alpha_2=-\frac{\tau(s_1)}{1 -\var \kappa(s_1) s_3 R(t,s_1)\cos s_2 }.
 \end{align}
Finally, for $i=3$ in (\ref{producto_igual_delta}) we find, on one hand, that
\begin{align} \nonumber
0&= \alpha_3 \left( 1 + \var s_3 R \left(\cos s_2(\bbv_2'\cdot\bbv_1) + \sin s_2 (\bbv_3'\cdot \bbv_1)\right)\right) \\ \nonumber
&\quad {} + \beta_3 \var s_3 \left( \pfrac{R}{s_1} \cos s_2 + R\sin s_2 (\bbv_3'\cdot\bbv_2)\right) \\ \label{et5}
&\quad {} + \gamma_3 \var s_3 \left( \pfrac{R}{s_1}\sin s_2 + R\cos s_2( \bbv_2'\cdot \bbv_3)\right), 
\end{align}
and, on the other hand, that
\begin{align}\nonumber
0&=-\beta_3\var s_3R\sin s_2 + \gamma_3 \var s_3 R \cos s_2
\\ \nonumber
1&=\beta_3\var R\cos s_2 + \gamma_3 \var  R \sin s_2.
\end{align}
where we deduce that,
\begin{align}
\beta_3=  \frac{\cos s_2}{\var  R(t,s_1)}, \quad \gamma_3=\frac{\sin s_2}{\var  R(t,s_1)}.  
\end{align}
 Therefore, from (\ref{Frenet}) and (\ref{et5}), we obtain that
  \begin{align}
  \alpha_3=-\frac{s_3}{R(t,s_1)\left(1-\var \kappa(s_1) s_3 R(t,s_1)\cos s_2 \right)}\pfrac{R}{s_1}(t,s_1).
  \end{align}
To sum up, we have obtained that  
\begin{align}\label{et12}
\pfrac{s_1}{ \bx^\var}&= \frac{1}{1-\var \kappa(s_1) s_3 R(t,s_1)\cos s_2 } \bbv_1(s_1),
\\\nonumber 
\pfrac{s_2}{ \bx^\var}&= -\frac{\tau(s_1)}{1-\var \kappa(s_1) s_3 R(t,s_1)\cos s_2 } \bbv_1(s_1) - \frac{\sin s_2}{\var s_3 R(t,s_1)}\bbv_2(s_1)
 \\ \label{et13}
&\quad {} + \frac{\cos s_2}{\var s_3 R(t,s_1)}\bbv_3(s_1),
\\ \nonumber
\pfrac{s_3}{ \bx^\var}&=-\frac{s_3}{R(t,s_1)\left(1-\var \kappa(s_1) s_3 R(t,s_1)\cos s_2 \right)}\pfrac{R}{s_1}(t,s_1)\bbv_1(s_1) \\ \label{et14}
&\quad {} + \frac{\cos s_2}{\var  R(t,s_1)}\bbv_2(s_1) 
+ \frac{\sin s_2}{\var  R(t,s_1)}\bbv_3(s_1),
\end{align}
and from the relation found in (\ref{relacion_der_temporal}), we deduce that
\begin{align}\label{ds_dt}
\pfrac{\bs}{t^\var}=- \frac{s_3}{R(t,s_1)}\pfrac{R}{t}(t,s_1)\bbv_3(s_1).
\end{align}

\subsection{Writing Navier-Stokes equations into the reference domain}

We are now in conditions to find the expressions of the fields in (\ref{NS_eps1})--(\ref{NS_eps2}) in the reference domain. Firstly, from (\ref{trans_dominioreferencia}), the chain rule and (\ref{ds_dt}),  we find that
\begin{align*}
\pfrac{u_i^\var}{t^\var}&=\pfrac{\bu^\var}{t^\var} \cdot \be_i= \pfrac{(u_k^\var\be_k)}{t^\var}\cdot\be_i= \pfrac{(u_k(\var)\bbv_k)}{t^\var}\cdot\be_i= \left( \pfrac{\left ( u_k(\var) \bbv_k \right )}{t}+ \pfrac{\left ( u_k(\var) \bbv_k \right )}{s_3}\pfrac{s_3}{t^\var}\right) \cdot \be_i
\\
&= \left(\pfrac{u_k(\var)}{t} - \frac{s_3}{R}\pfrac{R}{t}\pfrac{u_k(\var)}{s_3}  \right) (\bbv_k\cdot\be_i)= (D_t u_k) v_{ki},
\end{align*}
where $D_t$ is the operator defined by

\begin{align}
D_t:=\left(\pfrac{}{t} - \frac{s_3}{R}\pfrac{R}{t}\pfrac{}{s_3}\right).
\end{align}

The components of the non-linear term in (\ref{NS_eps1}), 
from (\ref{trans_dominioreferencia}) and the chain rule, can be written as follows,
\begin{align}\nonumber
\pfrac{u_i^\var}{x_j^\var}u_j^\var&= \pfrac{(\bu^\var\cdot\be_i)}{x_j^\var}(\bu^\var\cdot\be_j)= \pfrac{((u_k^\var\be_k)\cdot\be_i)}{x_j^\var}((u_m^\var\be_m)\cdot\be_j)
\\
&= \pfrac{((u_k(\var)\bbv_k)\cdot\be_i)}{x_j^\var}((u_m(\var)\bbv_m)\cdot\be_j)= \left( \pfrac{(u_k(\var) v_{ki})}{s_q} \pfrac{s_q}{x_j^\var}  \right) (u_m(\var)v_{mj}).
\end{align}

The Laplacian term in (\ref{NS_eps1}), from (\ref{trans_dominioreferencia}) and the chain rule, leads to
\begin{align}\nonumber
\Delta u_i^\var&= \Delta (\bu^\var \cdot \be_i)= \Delta (u_k^\var\be_k)\cdot \be_i=\Delta (u_k(\var)\bbv_k)\cdot \be_i= \pfrac{^2}{(x_j^\var)^2} \left(u_k(\var)v_{ki} \right)
\\
&= \pfrac{}{x_j^\var}\left( \pfrac{(u_k(\var) v_{ki})}{s_q}\pfrac{s_q}{x_j^\var} \right)=\pfrac{}{s_m}\left( \pfrac{(u_k(\var) v_{ki})}{s_q}\pfrac{s_q}{x_j^\var} \right)\pfrac{s_m}{x_j^\var}.
\end{align}

In the same way we obtain the components of the pressure gradient and of the volume forces as follows,
\begin{align*}
\pfrac{p}{x_i^\var}&= \pfrac{p}{s_q}\pfrac{s_q}{x_i^\var},
\\
(\mathbf{b}_0^\var)_i&= \mathbf{b}_0^\var\cdot \be_i= ((\mathbf{b}_0^\var)_k \be_k) \cdot \be_i= ((\mathbf{b}_0(\var))_k \bbv_k) \cdot \be_i=b_{0k}(\var)v_{ki},
\end{align*} 
where $b_{0k}(\var):= (\mathbf{b}_0(\var))_k$.

Finally, the incompressibility equation (\ref{NS_eps2}) in the reference domain, using (\ref{trans_dominioreferencia}) and the chain rule, has the following expression,
\begin{align*}
\mathrm{div} \, \bu^\var&= \pfrac{u_j^\var}{x_j^\var}= \pfrac{(\bu^\var\cdot \be_j)}{x_j^\var}= \pfrac{((u^\var_k\be_k)\cdot \be_j)}{x_j^\var}=\pfrac{((u_k(\var)\bbv_k)\cdot \be_j)}{x_j^\var}
\\&= \pfrac{(u_k(\var)v_{kj})}{s_q}\pfrac{s_q}{x_j^\var}.
\end{align*}

With these considerations, the incompressible Navier-Stokes equations in the reference domain 
can be written  as
\begin{align}
\nonumber
 &\displaystyle D_t(u_k(\varepsilon)v_{ki}) + \left( \frac{\partial(u_k(\varepsilon)v_{ki}) }{\partial s_q} \frac{\partial s_q}{\partial x_j^{\varepsilon}}\right) (u_m(\varepsilon)v_{mj}) \label{NavierStokesref1} 
 \\ \displaystyle 
 &\quad {} - \nu \frac{\partial}{\partial s_m}\left( \frac{\partial(u_k(\varepsilon)v_{ki}) }{\partial s_q} \frac{\partial s_q}{\partial x_j^{\varepsilon}}   \right) \frac{\partial s_m}{\partial x_j^{\varepsilon}} = - \frac{1}{\rho_0}\frac{\partial p(\varepsilon)}{\partial s_q} \frac{\partial s_q}{\partial x_i^{\varepsilon}} + b_{0k}(\varepsilon)v_{ki}, 
 \\ \displaystyle
 &\frac{\partial }{\partial s_q}(u_k(\varepsilon)v_{kj})\frac{\partial s_q}{\partial x_j^{\varepsilon}}=0. \label{NavierStokesref2}
 \end{align}

Let $\bn^\var= (\cos s_2) \bbv_2(s_1)+ (\sin s_2) \bbv_3(s_1)$ the outward unit normal vector at $s_3=1$. Then, from the boundary condition (\ref{bc-1}) at $s_3^\varepsilon=\varepsilon$, we have that
\begin{align}
\bu^\var=u_i^\var\be_i= u_i(\var)\bbv_i= \var \pfrac{R}{t}\left( (\cos s_2) \bbv_2(s_1)+ (\sin s_2) \bbv_3(s_1) \right),
\end{align}
hence, we obtain the following boundary conditions for the scaled components of velocity:
\begin{eqnarray}\label{bc_ref}
\left\{
\begin{array}{lll}
\label{cond1}
& u_1(\varepsilon)=0 &\ \mathrm{at} \ s_3=1,
\\[.8em] \displaystyle \label{cond2} & \displaystyle u_2(\varepsilon)=\varepsilon\frac{\partial R}{\partial t} \cos s_2 &\ \mathrm{at} \ s_3=1,\\[.8em] 
\label{cond3}
&\displaystyle u_3(\varepsilon)= \varepsilon\frac{\partial R}{\partial t} \sin s_2 &\ \mathrm{at} \ s_3=1.
\end{array} 
\right.
\end{eqnarray}

\section{Asymptotic expansion of the solution} \label{sec-3} \setcounter{equation}{0}
 
\subsection{Expansion of the solution on powers of $\var$}
 
Following \cite{Paloka}, 
we assume that the solution of (\ref{NavierStokesref1})-(\ref{NavierStokesref2}) admits a formal expansion on powers of $\varepsilon$, so the components of velocity and pressure fields can be written,
 \begin{align}\label{velocidad}
 u_k(\varepsilon)&= u_k^0+ \varepsilon u_k^1 + \varepsilon^2 u_k^2 +...  \\ \label{presion}
 p(\varepsilon)&= \frac{1}{\varepsilon^2}p^0 + \frac{1}{\varepsilon}p^1 + p^2 +... 
 \end{align}
 
We must remark that this assumption implies, as we shall see later, that the pressure gradient determines the velocity field. Other assumptions can be considered by choosing different order of 
$\var$ in the pressure and velocity fields in (\ref{velocidad})-(\ref{presion}), leading to different conclusions, but we consider that this is the most interesting case.

Substituting (\ref{velocidad})-(\ref{presion}) in the boundary conditions in the reference domain (see (\ref{bc_ref})), we obtain the following boundary conditions for the terms of the asymptotic expansion,
\begin{eqnarray} \label{bc_terms}
\left\{
\begin{array}{lll}
&\displaystyle u_1^k=0, \ k\geq 0, &\ \mathrm{at} \ s_3=1,  \\[.8em]
& u_2^0=u_3^0=0 &\ \mathrm{at} \ s_3=1,
\\[.8em] \displaystyle  & \displaystyle u_2^1=\frac{\partial R}{\partial t} \cos s_2, \quad u_3^1=\frac{\partial R}{\partial t} \sin s_2 &\ \mathrm{at} \ s_3=1,\\[.8em] 
&\displaystyle u_\alpha^k=0, \ k\geq 2, \ \alpha=2,3 &\ \mathrm{at} \ s_3=1.
\end{array} 
\right.
\end{eqnarray}

We need to write (\ref{et12})--(\ref{et14}) as expansions of $\var$. If we remark that
\begin{align*}
\frac{1}{a+\var b}= c_0 + c_1\var + c_2\var^2 + ...
\end{align*}
where $a,b\in\mathbb{R}$, such that $a\neq0$, then is easy to check that 
\begin{align*}
c_k= (-1)^k \frac{b^k}{a^{k+1}}, \ k\geq 0.
\end{align*}
Therefore, with $a=1$ and $b=-\kappa s_3 R\cos s_2 $ in (\ref{et12})--(\ref{et14}), we find that 
\begin{align}  \label{dsq_dx_j}
\pfrac{s_q}{ \bx^\var}\cdot \be_j=\pfrac{s_q}{ x^\var_j}&=\frac{1}{\var}d_{-1j}^q + d_{0j}^q + \var d_{1j}^q + \var^2 d_{2j}^q + ...
\end{align}
with $q=1,2,3$ and where,
\begin{align}
d_{-1j}^1&=0 , &d_{kj}^1&= (-1)^k \frac{b^k}{a^{k+1}}v_{1j}, \\
d_{-1j}^2&=- \frac{\sin s_2}{R s_3} v_{2j} + \frac{\cos s_2}{R s_3} v_{3j}, &d_{kj}^2&= (-1)^k \frac{(-\tau) b^k}{a^{k+1}}v_{1j}, \\
d_{-1j}^3&= \frac{\cos s_2}{R } v_{2j} + \frac{\sin s_2}{R } v_{3j}, &d_{kj}^3&= (-1)^{k+1} \frac{s_3}{R}\pfrac{R}{s_1} \frac{b^k}{a^{k+1}}v_{1j},
\end{align}
with $k\ge0.$ 

Also, we assume that the applied forces admit an asymptotic expansion of the form
\begin{align*}
b_{0k}(\var)= b_{0k}^0 + \var b_{0k}^1+ \var^2 b_{0k}^2+...
\end{align*}

We substitute (\ref{velocidad})-(\ref{presion})  into the equations (\ref{NavierStokesref1})-(\ref{NavierStokesref2}) and use (\ref{dsq_dx_j}). 
 Hence, we obtain the incompressible Navier-Stokes equations in the reference domain 
 in powers of $\var$:
\begin{align}
\nonumber
 &\displaystyle D_t((u_k^0+ \varepsilon u_k^1 + \varepsilon^2 u_k^2 +...)v_{ki}) 
 \\ \nonumber
 &\qquad {} + \left( \frac{\partial((u_k^0+ \varepsilon u_k^1 + \varepsilon^2 u_k^2 +...)v_{ki}) }{\partial s_q} \left(\frac{1}{\var}d_{-1j}^q + d_{0j}^q + \var d_{1j}^q + \var^2 d_{2j}^q + ...\right)\right) \left((u_m^0+ \varepsilon u_m^1 \right.
 \\ \displaystyle \nonumber
 & \qquad \left. {} + \varepsilon^2 u_m^2 +...)v_{mj}\right) - \nu \frac{\partial}{\partial s_m}\left( \frac{\partial((u_k^0+ \varepsilon u_k^1 + \varepsilon^2 u_k^2 +...)v_{ki}) }{\partial s_q}
  \left(\frac{1}{\var}d_{-1j}^q + d_{0j}^q + \var d_{1j}^q  \right. \right.
  \\ \nonumber
  & \qquad \left.  \left. {} + \var^2 d_{2j}^q + ...\right)   \right) \left(\frac{1}{\var}d_{-1j}^m + d_{0j}^m + \var d_{1j}^m + \var^2 d_{2j}^m + ...\right)
 \\ \nonumber
  &\quad = - \frac{1}{\rho_0}\frac{\partial (\frac{1}{\varepsilon^2}p^0 + \frac{1}{\varepsilon}p^1 + p^2 +...)}{\partial s_q} \left(\frac{1}{\var}d_{-1i}^q + d_{0i}^q + \var d_{1i}^q + \var^2 d_{2i}^q + ...\right) 
  \\ \label{NavierStokes_asymp} 
  &\qquad {} + \left(b_{0k}^0 + \var b_{0k}^1+ \var^2 b_{0k}^2+...\right)v_{ki}, 
 \\ \displaystyle \label{NavierStokes_asymp2}
 &\frac{\partial }{\partial s_q}((u_k^0+ \varepsilon u_k^1 + \varepsilon^2 u_k^2 +...)v_{kj})\left(\frac{1}{\var}d_{-1j}^q + d_{0j}^q + \var d_{1j}^q + \var^2 d_{2j}^q + ...\right)=0. 
 \end{align}

\subsection{Asymptotic expansion of the flow}

In the next proposition we shall show some conditions satisfied by the asymptotic expansion of the flow, that will be used later to characterize the terms of the asymptotic expansion of the solution. Firstly, let $Q^\var(t,s)$ denotes the flow at position $s_1=s$ and at time $t$, defined in the original domain by 
\begin{align}\label{flow}
Q^\varepsilon(t,s):=\int_{s_1=s}\mathbf{u^\varepsilon}\cdot \bbv_1\, dA.
\end{align}
where $\displaystyle \int_{s_1=s} \phi\, dA$ represents the surface integral of $\phi$ on 
the transversal section of $\hat{\Omega}_t^\varepsilon$ at $s_1=s$.

Using the mapping (\ref{aplicacion}), we define the flow in the reference domain by $Q(\var)$, that is
\begin{align} \label{flow_eps2}
Q^\varepsilon(t,s_1)=\var^2 Q(\var)(t, s_1),
\end{align}
 where
\begin{align} \label{flow_ref}
 Q(\var)=R^2\int_{0}^{2\pi}\int_{0}^{1} s_3 u_1(\var)\, ds_3 ds_2.
\end{align}

We also define the cross-sectional area in the original domain by
\begin{align}\label{area_eps}
A^\var:=\var^2A_0= \var^2 \pi R^2,
\end{align}
where $A_0$ denotes the cross-sectional area in the reference domain, $A_0=\pi R^2$.

\begin{proposition}
Let us consider a fluid inside the curved pipe $\hat{\Omega}_t^\varepsilon$, which movement is described by the incompressible Navier-Stokes equations (\ref{NS_eps1})--(\ref{NS_eps2}) with the boundary condition (\ref{bc-1}). Let us assume that there exists an asymptotic expansion of the form (\ref{velocidad})--(\ref{presion}) of the problem in the reference domain.  Then there exists an asymptotic expansion for the scaled flow in the reference domain of the form
\begin{align*}
Q(\varepsilon) = Q^0 + \varepsilon Q^1 + \varepsilon^2Q^2+...
\end{align*}
where the term $Q^k$ is defined by
\begin{align}\label{flow_k}
Q^k= R^2\int_{0}^{2\pi}\int_{0}^{1} s_3 u_1^k\, ds_3 ds_2.
\end{align}

Moreover, the following relations hold:
\begin{align}\label{relaciones_flow}
\frac{\d Q^0}{\d s_1} + \frac{\d A^0}{\d t}=0,
\qquad
\frac{\d Q^k}{\d s_1}=0 \quad (k\geq 1).
\end{align}
\end{proposition}
\begin{proof}
Let $\tilde{\Omega}^\epsilon_t$ be a portion of the original domain $\hat{\Omega}_t^\varepsilon$ between $s_1=a$ and $s_1=b$ $(a<b)$. From $(\ref{NS_eps2})$ and the Gauss Theorem, we deduce that
\begin{align}\nonumber
0&=\int_{\tilde{\Omega}^\epsilon_t} \mathrm{div} \, \mathbf{u}^\varepsilon\, dV=\int_{\d\tilde{\Omega}^\epsilon_t} \mathbf{u^\varepsilon}\cdot \mathbf{n^\varepsilon}\, dA 
 \\ \label{et6}
 &=\int_{s_1=a} \mathbf{u^\varepsilon}\cdot \mathbf{n^\varepsilon}\, dA+ \int_{s_1=b} \mathbf{u^\varepsilon}\cdot \mathbf{n^\varepsilon}\, dA +\int_{s_3^\varepsilon=\varepsilon R(t,s_1)} \mathbf{u^\varepsilon}\cdot \mathbf{n^\varepsilon}\, dA.
\end{align}

At the beginning and end of  $\tilde{\Omega}^\epsilon_t$ we have that
\begin{align*}
\mathbf{u^\varepsilon}\cdot \mathbf{n^\varepsilon}&= (u_k(\varepsilon)\mathbf{v}_k)\cdot (-\mathbf{v}_1)=-u_1(\varepsilon) \ &\textrm{at} \ s_1=a, \\
\mathbf{u^\varepsilon}\cdot \mathbf{n^\varepsilon}&= (u_k(\varepsilon)\mathbf{v}_k)\cdot (\mathbf{v}_1)=u_1(\varepsilon) \ &\textrm{at} \ s_1=b.
\end{align*}

Therefore, from (\ref{flow}) and (\ref{et6}) we obtain 
\begin{align}\label{et7}
0=-Q^\varepsilon(t, a) + Q^\varepsilon(t, b) + \int_{s_3^\varepsilon=\varepsilon R(t,s_1)} \mathbf{u^\varepsilon}\cdot \mathbf{n^\varepsilon}\, dA.
\end{align}

At $s_3^\varepsilon=\varepsilon R$, we must consider continuity between the fluid and the wall of the pipe displacements (see (\ref{bc-1})), hence
\begin{align*}
\mathbf{u^\varepsilon}\cdot \mathbf{n^\varepsilon}= \frac{\d}{\d t}[ \mathbf{c}(s_1) + \varepsilon R(t,s_1) (\cos s_2^\varepsilon \mathbf{v}_2(s_1) + \sin s_2^\varepsilon \mathbf{v}_3(s_1))  ] \cdot \mathbf{n}^\varepsilon = \varepsilon \frac{\d R}{\d t},
\end{align*}
and then,
\begin{align*}
\int_{s_3^\varepsilon=\varepsilon R(t,s_1)} \mathbf{u^\varepsilon}\cdot \mathbf{n^\varepsilon}\, dA&= \int_{s_1=a}^{s_1=b}\int_{s_2=0}^{s_2=2\pi} \varepsilon^2 R\frac{\d R}{\d t}\,  ds_2ds_1
= \int_{a}^{b} 2\pi \varepsilon^2 R \frac{\d R}{\d t}\, ds_1.
\end{align*}

Substituting this in (\ref{et7}) and dividing the expression by $b-a$, we obtain that 
\begin{align*}
0=\frac{Q^\varepsilon(t,b)-Q^\varepsilon(t,a)}{b-a} + \frac{1}{b-a}\int_{a}^{b} 2\pi \varepsilon^2 R \frac{\d R}{\d t}\, ds_1. 
\end{align*}

Taking the limit when $b$ tends to $a$ and using (\ref{area_eps}), we obtain the following relation,
\begin{align*}
0=\frac{\d Q^\varepsilon}{\d s_1} + 2\pi\varepsilon^2 R \frac{\d R}{\d t}= \frac{\d Q^\varepsilon}{\d s_1} + \frac{\d A^\varepsilon}{\d t}. 
\end{align*}

Now, since $A^0=A^\varepsilon/\varepsilon^2=\pi R^2$ and 
$Q(\varepsilon)=Q^\varepsilon/\varepsilon^2$, we deduce that
\begin{align}\label{et8}
\frac{\d Q(\varepsilon)}{\d s_1} + \frac{\d A^0}{\d t}=0.
\end{align}

On the other hand, taking into account the expansion for $u_1(\varepsilon)$ in (\ref{velocidad}) and 
(\ref{flow_ref}), we can deduce that there exists an asymptotic expansion of the form
\begin{align}\label{desarrollo_flow}
Q(\varepsilon) = Q^0 + \varepsilon Q^1 + \varepsilon^2Q^2+...
\end{align}
where 
\begin{align*}
Q^k= R^2\int_{0}^{2\pi}\int_{0}^{1} s_3 u_1^k\, ds_3 ds_2.
\end{align*}

Finally, upon substitution of (\ref{desarrollo_flow}) in (\ref{et8}), we conclude that 
\begin{align*}
\frac{\d Q^0}{\d s_1} + \frac{\d A^0}{\d t}=0,
\qquad
\frac{\d Q^k}{\d s_1}=0 \quad (k\geq 1).
\end{align*}
\end{proof}

\begin{remark}
The previous result is a direct consequence of the law of conservation of mass. Similar results can be found in previous works, for instance, see \cite{Quarteroni}.
\end{remark}

\subsection{Identifying the terms of the asymptotic expansion of the solution}

In order to identify some of the terms of the asymptotic expansion proposed in (\ref{velocidad})-(\ref{presion}), we shall group the terms multiplied by the same power of $\varepsilon$ in (\ref{NavierStokes_asymp})--(\ref{NavierStokes_asymp2}), obtaining in this way new equations, easier than the original one, that can be solved to identify the mentioned terms of the asymptotic expansion. With this aim, we shall recall a result from \cite{Temam} (Theorem 2.4), that will be used in the following.

\begin{theorem}\label{Teman}
Let $\Omega$ be an open bounded set of class $\mathcal{C}^2$ in $\mathbb{R}^n$ and $\Gamma=\d\Omega$. Let there be given $\mathbf{f}\in H^{-1}(\Omega), g\in L^2(\Omega), \bvarphi\in H^{1/2}(\Gamma)$, such that
\begin{align*}
\int_{\Omega}g\, d\mathbf{x}=\int_{\Gamma}\bvarphi\cdot \mathbf{n}\, d\Gamma,
\end{align*}

%(where $\mathbf{v}$ represents a test function).
Then there exists $\mathbf{u}\in H^1(\Omega), p\in L^2(\Omega)$, which are solutions of the Stokes problem 
\begin{equation*}
\left \{ \begin{array}{l}
-\nu\Delta\mathbf{u} + \nabla p=\mathbf{f} \ \textrm{in} \ \Omega,
\\ \mathrm{div}\, \mathbf{u}=g \ \textrm{in} \ \Omega,
\\\displaystyle 
\mathbf{u} = \bvarphi \quad \textrm{on $\Gamma$}.
\end{array} \right .
\end{equation*}
$\mathbf{u}$ is unique and $p$ is unique up to the addition of a constant.
\end{theorem}

Let us introduce the local cartesian 
coordinates of the cross section of the pipe at $s_1$, as the points $\mathbf{z}\in\mathbb{R}^2$ defined by
\begin{align}\label{def_z}
\mathbf{z}= (z_2, z_3) = (s_3 \cos s_2, s_3 \sin s_2),
\end{align}  
and let $\omega=\{(z_2,z_3) \in \mathbb{R}^2 / z_2^2+z_3^2 < 1 \}$. We can prove now the following theorem, where the first terms of the asymptotic expansion are identified.

\begin{theorem}\label{terminos_identificados}
Let us assume that there exists an asymptotic expansion of the form (\ref{velocidad})-(\ref{presion}). Then:
\begin{enumerate}[label={{(\roman*)}}, leftmargin=0em ,itemindent=3em]
\item The term of order zero of velocity, $\mathbf{u}^0$, verifies 
\begin{align}\label{u_1^0}
  u_1^0&= \frac{R^2}{4\rho_0 \nu} \frac{\partial p^0}{\partial s_1} (s_3^2 -1), \\ \label{U^0}
  u_2^0&=u_3^0=0,
\end{align}
while zeroth order term of pressure, $p^0$, is the solution of the problem,
\begin{equation}\label{p^0}
 \frac{\d}{\d s_1} \left( R^4 \frac{\d p^0}{\d s_1} \right) = 16 \nu \rho_0 R \frac{\d R}{\d t},
\end{equation}
with suitable boundary conditions.
\item The components of the first order term of velocity, $\mathbf{u}^1$, are 
  \begin{align}
  \label{u^1_1}
   u^1_1&= \left [ \frac{3R^3\kappa s_3 \cos s_2}{16 \nu \rho_0} \frac{\d p^0}{\d s_1}  
  + \frac{R^2}{4 \nu \rho_0}\frac{\d p^1}{\d s_1}\right ](s_3^2- 1), \\
 u^1_2&= \frac{s_3 R}{16 \rho_0 \nu} \left[ 2 \frac{\d}{\d s_1} \left (R^2 \frac{\d p^0}{\d s_1} \right ) -R^2s_3^2 \frac{\d^2 p^0}{\d s_1^2} \right] \cos s_2,\\ 
   u^1_3&= \frac{s_3 R}{16 \rho_0 \nu} \left[ 2 \frac{\d}{\d s_1} \left (R^2 \frac{\d p^0}{\d s_1} \right ) -R^2s_3^2 \frac{\d^2 p^0}{\d s_1^2} \right]  \sin s_2.
  \end{align}
The first order term of pressure, $p^1$, is the solution of the  problem, 
\begin{equation}\label{edp_p1}
  \frac{\d}{\d s_1} \left(R^4 \frac{\d p^1}{\d s_1}\right)=0,
\end{equation}
with the appropriate boundary conditions.

\item The first component of the second order term of velocity, $\mathbf{u}^2$, is given by 
\begin{align}\label{u_1^2}\displaystyle  \nonumber
u_1^2&= \frac{R^2}{16} \left[   \frac{R^2}{4 \rho_0 \nu^2} \frac{\d^2 p^0}{\d t\d s_1 } - \frac{R^4}{16\rho_0^2\nu^3}\frac{\d p^0}{\d s_1}\frac{\d^2p^0}{\d s_1^2}
- \frac{R^2}{2 \rho_0 \nu}\frac{\d^3 p^0}{\d s_1 ^3} 
 + \frac{11 \kappa^2 R^2}{8 \rho_0 \nu}\frac{\d p^0}{\d s_1} \right] (s_3^4 - 1) 
  \\&\displaystyle \qquad \nonumber
{} + \frac{R^2}{4}  \left[-\frac1{4\rho_0 \nu^2}\frac{\d}{\d t}\left(R^2 \frac{\d p^0}{\d s_1}  \right) 
+ \frac{R^2}{16 \rho_0^2 \nu^3}\frac{\d p^0}{\d s_1} \frac{\d}{\d s_1} \left( R^2 \frac{\d p^0}{\d s_1}  \right) 
\right.\\& \displaystyle \left. \qquad \nonumber
{} + \frac1{4 \rho_0 \nu} \frac{\d^2}{\d s_1^2} \left(R^2 \frac{\d p^0}{\d s_1}\right)
 -\frac{7 \kappa^2 R^2}{16 \rho_0 \nu}\frac{\d p^0}{\d s_1} +\frac1{\rho_0 \nu} \frac{\d p^2_0}{\d s_1} - \frac{b_{01}}{\nu}\right] (s_3^2-1 ) 
 \\&\displaystyle \qquad \nonumber
 {}  + \frac{R^6}{1152 \rho_0^2 \nu^3} \frac{\d p^0}{\d s_1}\frac{\d^2 p^0}{\d s_1^2}(s_3^6-1)
+ \frac{3\kappa R^3 }{16 \rho_0 \nu}\frac{\d p^1}{\d s_1} (s_3^3- s_3)\cos s_2 
 \\& \displaystyle \qquad
{} + \frac{5 \kappa^2 R^4}{64 \rho_0 \nu} \frac{\d p^0}{\d s_1} (s_3^4- s_3^2)\cos(2 s_2),
 \end{align}
and the second order term of pressure, $p^2$,  is
\begin{equation} \label{et-p2}
 p^2 = - \frac{R^2}{4} \frac{\d ^2 p^0}{\d s_1^2}s_3^2 + p_0^2(t,s_1),
  \end{equation}
where $p_0^2(t,s_1)$ is the solution, with the adequate boundary conditions, of the problem
\begin{align} \nonumber
\displaystyle \frac{\d}{\d s_1} \left(R^4 \frac{\d p_0^2}{\d s_1}\right) &= \frac{\d}{\d s_1} \left[ - \frac{3 R^8}{64 \rho_0 \nu^2}\frac{\d p^0}{\d s_1}\frac{\d^2 p^0}{\d s_1^2}
 - \frac{R^6}{12} \frac{\d^3 p^0}{\d s_1^3} \right. - \frac{\kappa^2 R^6}{48} \frac{\d p^0}{\d s_1} +\frac{R^5}{2\nu} \frac{\d R}{\d t} \frac{\d p^0}{\d s_1} 
\\& \displaystyle  \qquad \nonumber
 {} - \frac{R^7}{8 \rho_0 \nu^2} \frac{\d R}{\d s_1} \left(\frac{\d p^0}{\d s_1}\right)^2 -\frac{R^4}{2}\left(\frac{\d R}{\d s_1}\right)^2\frac{\d p^0}{\d s_1} 
  \left.
 -\frac{R^5}{2}\frac{\d^2 R}{\d s_1^2}\frac{\d p^0}{\d s_1} 
  \right.\\& \displaystyle \left. \qquad \label{eq-dif-p02}
{} - R^5 \frac{\d R}{\d s_1}\frac{\d^2 p^0}{\d s_1^2} +  \frac{R^6}{6\nu} \frac{\d^2p^0}{\d t\d s_1}   + R^4 \rho_0 b_{01} \right].
 \end{align}

Let us consider the local cartesian 
coordinates at cross section of the pipe at $s_1$, defined by $\bz= (z_2, z_3) = (s_3 \cos s_2, s_3 \sin s_2)$ 
(and then, $(s_3, s_2)$ are the local polar coordinates at the same cross section). Let be 
$\mathbf{U}^2 = (u^2_2, u^2_3)$. Then $(\mathbf{U}^2, p^3 )$ is the solution of the following problem 
\begin{equation}
\label{eqU2}
\left \{ \begin{array}{lcl}
\displaystyle \Delta_z \mathbf{U}^2 = \frac{R}{\rho_0 \nu}\nabla_z p^3 + \mathbf{F} &\textrm{in} &\omega, 
\\  \displaystyle \mathrm{div}_{\bz}\, \mathbf{U}^2 = g &\textrm{in} &\omega, 
\\  \displaystyle 
\mathbf{U}^2 = \mathbf{0} &\textrm{on} &\d \omega, 
\end{array} \right .
\end{equation}
where $\omega = \{ (z_2, z_3) / z_2^2 + z_3^2 < 1 \}$ and the fields $g$ and $\mathbf{F}$ are defined, respectively, by
\begin{align}\nonumber
&g= \left(  - \frac{\kappa R^4}{2 \rho_0 \nu}\pfrac{^2p^0}{s_1^2} - \frac{3 \kappa' R^4}{16\rho_0 \nu}\pfrac{p^0}{s_1} \right)z_2\left( z_2^2 + z_3^2\right)  -\frac{3\kappa\tau R^4}{16\rho_0\nu}\pfrac{p^0}{s_1}z_3\left( z_2^2 + z_3^2\right) 
\\ \nonumber
&\qquad {} + \left(  \frac{9\kappa R^3}{8 \rho_0 \nu}\pfrac{R}{s_1}\pfrac{p^0}{s_1} + \frac{9 \kappa R^4}{16 \rho_0 \nu} \pfrac{^2p^0}{s_1^2} + \frac{3 \kappa' R^4}{16 \rho_0 \nu} \pfrac{p^0}{s_1}             \right)z_2
\\ \label{funcion_g}
& \qquad {} + \frac{3\kappa\tau R^4}{16\rho_0\nu}\pfrac{p^0}{s_1} z_3 -\frac{R^3}{4\rho_0\nu}\pfrac{^2p^1}{s_1^2}\left( z_2^2 + z_3^2\right) + \frac{R}{4\rho_0\nu}\pfrac{}{s_1}\left( R^2 \pfrac{p^1}{s_1} \right),
\end{align}
\begin{align}  \nonumber
\mathbf{F}&= \left( \frac{\kappa R^6}{16 \rho_0^2 \nu^3} \left(\frac{\d p^0}{\d s_1}\right)^2 \left( (z_2^2 + z_3^2)^2 +1 \right)   + \frac{\kappa' R^4}{4 \rho_0 \nu}\frac{\d p^0}{\d s_1} + \frac{5 R^2 \kappa}{8 \rho_0 \nu} \frac{\d}{\d s_1}\left(R^2 \frac{\d p^0}{\d s_1} \right) \right.
 \\ \nonumber
 & \qquad {} + \left(- \frac{\kappa R^6}{8 \rho^2_0 \nu^3} \left(\frac{\d p^0}{\d s_1}  \right)^2 - \frac{9R^4 \kappa}{16 \rho_0 \nu}\frac{\d^2 p^0}{\d s_1^2} - \frac{\kappa'R^4}{4\rho_0 \nu}\frac{\d p^0}{\d s_1} \right){(z_2^2 + z_3^2)} -\frac{\kappa R^4 }{8 \rho_0 \nu} \frac{\d^2 p^0}{\d s_1^2}z_2^2  
\\ 
&\qquad \left.
 {} - \frac{R^2}{\nu}b_{02}, -\frac{\kappa\tau R^4}{4 \rho_0 \nu} \frac{\d p^0}{\d s_1} (z_2^2 + z_3^2 -1) -\frac{2R^4\kappa}{16 \rho_0\nu}\frac{\d^2p^0}{\d s_1^2}z_2z_3 
   -\frac{R^2}{\nu}b_{03}   \right). \label{funcion_F}
\end{align}

Problem (\ref{eqU2}) has a unique solution $(\mathbf{U}^2,p^3)$, with $\mathbf{U}^2$ unique and $p^3$ unique up to a function depending on $t$ and $s_1$, that can be computed explicitly (see appendix \ref{apend.U2}).
\end{enumerate}
\end{theorem}
\begin{proof}
After substitution of (\ref{velocidad})--(\ref{presion}) in (\ref{NavierStokesref1})--(\ref{NavierStokesref2}), we have obtained equations 
(\ref{NavierStokes_asymp})--(\ref{NavierStokes_asymp2}), and we proceed now to identify the terms of the asymptotic expansion.
 
 \begin{enumerate}[label={{(\roman*)}}, leftmargin=0em ,itemindent=3em]
  \item Grouping the terms multiplied by $\var^{-3}$ in the equation (\ref{NavierStokes_asymp}), we find these two equations related with the zeroth-order term of pressure,
\begin{align*}
-\frac{\sin s_2}{s_3}\frac{\d p^0}{\d s_2} + (\cos s_2) \frac{\d p^0}{\d s_3} =0, \qquad 
\frac{\cos s_2}{s_3}\frac{\d p^0}{\d s_2} + (\sin s_2) \frac{\d p^0}{\d s_3} =0.
\end{align*}
Therefore, it is clear that $\displaystyle \frac{\d p^0}{\d s_2} = \frac{\d p^0}{\d s_3} =0$, so 
\begin{align}\label{p0_dependencia}
p^0=p^0(t,s_1).
\end{align}

This is, the zeroth-order term of pressure does not depend on the cross-sectional variables and only depends on time and on the point $s_1$ of the midle line of the curved pipe. 

If we group now the terms multiplied by $\varepsilon^{-2}$ in the equation (\ref{NavierStokes_asymp}), we obtain the equations
\begin{align}\label{et9}
\frac{1}{(Rs_3)^2}\frac{\d^2 u_1^0}{\d s_2^2} + \frac{1}{R^2s_3}\frac{\d u_1^0}{\d s_3} + \frac{1}{R^2}\frac{\d^2u_1^0}{\d s_3^2}&=\frac{1}{\nu\rho_0}\frac{\d p^0}{\d s_1},\\ \label{et10}
\frac{1}{(Rs_3)^2}\frac{\d^2 u_2^0}{\d s_2^2} + \frac{1}{R^2s_3}\frac{\d u_2^0}{\d s_3} + \frac{1}{R^2}\frac{\d^2u_2^0}{\d s_3^2}&= \frac{1}{\nu\rho_0}\left(-\frac{\sin s_2}{R s_3}\frac{\d p^1}{\d s_2} + \frac{\cos s_2}{R}\frac{\d p^1}{\d s_3}\right),
\\ \label{et11}
\frac{1}{(Rs_3)^2}\frac{\d^2 u_3^0}{\d s_2^2} + \frac{1}{R^2s_3}\frac{\d u_3^0}{\d s_3} + \frac{1}{R^2}\frac{\d^2u_3^0}{\d s_3^2}&= \frac{1}{\nu\rho_0}\left(\frac{\cos s_2}{R s_3}\frac{\d p^1}{\d s_2} + \frac{\sin s_2}{R}\frac{\d p^1}{\d s_3}\right).
\end{align}

Using the change of variable (\ref{def_z}) in (\ref{et9}), and taking into account the boundary condition in (\ref{bc_terms}), we obtain the following problem for the axial component of the zeroth-order term of velocity:
\begin{equation}\label{pb_u10}
\left \{ \begin{array}{ll}
\displaystyle\Delta_\mathbf{z}u_1^0 \displaystyle=\frac{R^2}{\nu\rho_0}\frac{\d p^0}{\d s_1} &\textrm{in} \ \omega, 
\\ u_1^0=0 \ &\textrm{on} \ \d \omega. 
\end{array} \right .
\end{equation}

The problem (\ref{pb_u10}) has a unique solution, which expression is
\begin{equation*}
u_1^0= \frac{R^2}{4\rho_0 \nu} \frac{\partial p^0}{\partial s_1} (s_3^2 -1).
\end{equation*}

Now, from the first relation in (\ref{relaciones_flow}), we have that
\begin{align*}
\frac{\d}{\d s_1} \left(R^2\int_{0}^{2\pi} \int_0^1 s_3 u_1^0\, ds_3 ds_2\right)= -2\pi R \frac{\d R}{\d t}.
\end{align*}

Hence, using the expression for $u_1^0$ that we have obtained, we deduce that the left-hand side of last equality verifies 
\begin{align*}
\frac{\d}{\d s_1} \left(R^2\int_{0}^{2\pi} \int_0^1 \frac{R^2}{4\rho_0 \nu} \frac{\partial p^0}{\partial s_1} s_3(s_3^2 -1)\, ds_3 ds_2\right)
= \frac{\d}{\d s_1} \left(-\frac{2 \pi R^4}{16\rho_0 \nu}\frac{\partial p^0}{\partial s_1}\right)
\end{align*}

Since $p^0$ does not depend on the cross-sectional variables (see (\ref{p0_dependencia})),  we obtain that $p^0$ satisfies the following equation,
\begin{equation*}
 \frac{\d}{\d s_1} \left( R^4 \frac{\d p^0}{\d s_1} \right) = 16 \nu \rho_0 R \frac{\d R}{\d t}, 
\end{equation*}
that has a unique solution with the appropriate initial and boundary conditions.

If we now group the terms multiplied by $\varepsilon^{-1}$ in the equation (\ref{NavierStokes_asymp2}), we find that
\begin{align*}
-\frac{\sin s_2}{s_3}\frac{\d u_2^0}{\d s_2} + (\cos s_2) \frac{\d u_2^0}{\d s_3} + \frac{\cos s_2}{s_3}\frac{\d u_3^0}{\d s_2} + (\sin s_2) \frac{\d u_3^0}{\d s_3}=0.
\end{align*}

Using the change of variable (\ref{def_z}) in this last equation and in (\ref{et10})--(\ref{et11}), and considering the boundary conditions in (\ref{bc_terms}), the cross-sectional components of the zeroth-order term of velocity, denoted by $\mathbf{U}^0=(u_2^0,u_3^0)$, and the first order term of pressure, $p^1$, are solution of the problem,
\begin{equation*}
\left \{ \begin{array}{ll}
	\displaystyle\Delta_\mathbf{z}\mathbf{U^0}=\frac{R}{\nu\rho_0}\nabla_\mathbf{z}p^1 
	&\textrm{in}\ \omega, 
\\ \mathrm{div}_\mathbf{z}\mathbf{U^0}=0 &\textrm{in}\ \omega, 
\\\displaystyle 
\mathbf{U^0} = \mathbf{0} &\textrm{on}\ \d \omega,
\end{array} \right .
\end{equation*}
Applying Theorem \ref{Teman}, this problem has a unique solution 
(up to an arbitrary function depending only on $t$ and $s_1$, for the pressure term), and this solution is
\begin{align*}
{u_2^0=u_3^0=0, \quad p^1=p^1(t,s_1)}.
\end{align*}

\item Grouping the terms multiplied by $\varepsilon^{-1}$ in the equation (\ref{NavierStokes_asymp}), we obtain 
	
\begin{align}\label{et15}
\frac{1}{(Rs_3)^2}\frac{\d^2 u_1^1}{\d s_2^2} + \frac{1}{R^2s_3}\frac{\d u_1^1}{\d s_3} + \frac{1}{R^2}\frac{\d^2u_1^1}{\d s_3^2}&=\frac{1}{\nu\rho_0}\left( \frac{\d p^1}{\d s_1} + R\kappa s_3\cos s_2\frac{\d p^0}{\d s_1}\right) + \frac{\kappa \cos s_2}{R}\frac{\d u^0_1}{\d s_3},
\\ \label{et16}
		\frac{1}{(Rs_3)^2}\frac{\d^2 u_2^1}{\d s_2^2} + \frac{1}{R^2s_3}\frac{\d u_2^1}{\d s_3} + \frac{1}{R^2}\frac{\d^2u_2^1}{\d s_3^2}&= \frac{1}{\nu\rho_0}\left(-\frac{\sin s_2}{R s_3}\frac{\d p^2}{\d s_2} + \frac{\cos s_2}{R}\frac{\d p^2}{\d s_3}\right),
		\\ \label{et17}
		\frac{1}{(Rs_3)^2}\frac{\d^2 u_3^1}{\d s_2^2} + \frac{1}{R^2s_3}\frac{\d u_3^1}{\d s_3} + \frac{1}{R^2}\frac{\d^2u_3^1}{\d s_3^2}&= \frac{1}{\nu\rho_0}\left(\frac{\cos s_2}{R s_3}\frac{\d p^2}{\d s_2} + \frac{\sin s_2}{R}\frac{\d p^2}{\d s_3}\right).
		\end{align}
		
Using (\ref{u_1^0}) in (\ref{et15}), and then the change of variable (\ref{def_z}) and the boundary condition (\ref{bc_terms}), we obtain that $u_1^1$ is the unique solution of the problem 
	\begin{equation}\label{pb_u11}
	\left \{ \begin{array}{ll}
	\displaystyle\Delta_\mathbf{z}u_1^1 \displaystyle=\frac{R^2}{\nu\rho_0}\left( \frac{\d p^1}{\d s_1} + \frac{3R\kappa}{2} z_2\frac{\d p^0}{\d s_1}\right) &\textrm{in} \ \omega, 
	\\ u_1^1=0 \ &\textrm{on} \ \d \omega.
	\end{array} \right .
	\end{equation}

Now, it is easy to check (for example, by substitution in (\ref{et15})) that 
the unique solution is 
	\begin{align*}
	u^1_1= \left [ \frac{3R^3\kappa s_3 \cos s_2}{16 \nu \rho_0} \frac{\d p^0}{\d s_1}  
	+ \frac{R^2}{4 \nu \rho_0}\frac{\d p^1}{\d s_1}\right ](s_3^2- 1).
	\end{align*}
	
From the second relation in (\ref{relaciones_flow}), for $k=1$, we have that
	\begin{align*}
	\frac{\d}{\d s_1} \left(R^2\int_{0}^{2\pi} \int_0^1 s_3 u_1^1\, ds_3 ds_2\right)= 0,
	\end{align*}
and, substituting the expression of $u_1^1$, we obtain that
	\begin{align*}
	0&= \frac{\d}{\d s_1} \left(R^2\int_{0}^{2\pi} \int_0^1 \left [ \frac{3R^3\kappa s_3 \cos s_2}{16 \nu \rho_0} \frac{\d p^0}{\d s_1}  
	+ \frac{R^2}{4 \nu \rho_0}\frac{\d p^1}{\d s_1}\right ] s_3(s_3^2 -1)\, ds_3 ds_2\right)\\
	&=\frac{\d}{\d s_1} \left(R^2\int_{0}^{2\pi} \int_0^1 \frac{R^2}{4\rho_0 \nu} \frac{\partial p^1}{\partial s_1} s_3(s_3^2 -1)\, ds_3ds_2\right)\\
	&= \frac{\d}{\d s_1} \left(-\frac{2\pi R^4}{16\rho_0 \nu}\frac{\partial p^1}{\partial s_1}\right),
	\end{align*}
so we conclude that $p^1$ 
	(remember that it is only function of $t$ and $s_1$) is solution of
	\begin{equation*}
	\frac{\d}{\d s_1} \left( R^4 \frac{\d p^1}{\d s_1} \right) =0,
	\end{equation*}
with suitable initial and boundary conditions.

If we now group the terms multiplied by $\varepsilon^{0}$  in (\ref{NavierStokes_asymp2}), we obtain 
	\begin{align*}
	& -\frac{\sin s_2}{s_3}\frac{\d u_2^1}{\d s_2} + (\cos s_2)\frac{\d u_2^1}{\d s_3} + \frac{\cos s_2}{s_3}\frac{\d u_3^1}{\d s_2} + (\sin s_2) \frac{\d u_3^1}{\d s_3} \\
	&\qquad{} + R\left (\frac{\d u_1^0}{\d s_1}-\tau\frac{\d u_1^0}{\d s_2}
	- \frac{s_3}{R}\frac{\d R}{\d s_1}\frac{\d u_1^0}{\d s_3}\right )=0.
	\end{align*}
	
Let be $\mathbf{U^1}=(u^1_2,u^1_3)$. Applying the change of variable (\ref{def_z}) to the previous equation, to (\ref{et16})--(\ref{et17}) and taking into account the boundary conditions (\ref{bc_terms}), we find that $(\mathbf{U^1}, p^2)$ is solution of the problem
\begin{equation} \label{problema_U1}
\left \{ \begin{array}{l}
\displaystyle\Delta_\mathbf{z}\mathbf{U^1}=\frac{R}{\nu\rho_0}\nabla_\mathbf{z}p^2 
\quad \textrm{in}\ \omega,
\\ \displaystyle \mathrm{div}_\mathbf{z}\mathbf{U^1}=\frac{R}{4 \rho_0 \nu}\left(\frac{\d}{\d s_1}\left(R^2 \frac{\d p^0}{\d s_1}\right) - (z_2^2+z_3^2)R^2\frac{\d^2 p^0}{\d s_1^2} \right)=:g^1 
\quad \textrm{in}\ \omega, 
\\ \displaystyle 
\mathbf{U^1} = \frac{\d R}{\d t}(\cos s_2, \sin s_2)=:\mathbf{\bphi^1} \quad \textrm{on} \ \d \omega.
\end{array} \right.
\end{equation}

Theorem \ref{Teman} ensures the existence and uniqueness of solution of this problem 
(up to an arbitrary function depending only on $t$ and $s_1$, for the pressure term) if the compatibility condition given by 
\begin{align}\label{cond_comp_U1}
\int_{\omega}g^1 =\int_{\d\omega}\bphi^1\cdot \mathbf{n},
\end{align}
where $\mathbf{n}=(\cos s_2, \sin s_2)$ is the unit outward normal vector on $\d\omega$, is fulfilled. On one hand, we have
\begin{align*}
\int_{\omega}g^1&= \int_{\omega} \frac{R}{4 \rho_0 \nu}\left(\frac{\d}{\d s_1}\left(R^2 \frac{\d p^0}{\d s_1}\right) - (z_2^2+z_3^2)R^2\frac{\d^2 p^0}{\d s_1^2} \right)\, dz_2 dz_3
\\
&=\int_{0}^{1}\int_{0}^{2\pi} \frac{s_3R}{4 \rho_0 \nu}\left(\frac{\d}{\d s_1}\left(R^2 \frac{\d p^0}{\d s_1}\right) - s_3^2R^2\frac{\d^2 p^0}{\d s_1^2} \right)\, ds_2 ds_3
\\
&=\frac{2\pi R}{4 \rho_0 \nu} \int_{0}^{1} \left ( s_3\frac{\d}{\d s_1}\left(R^2 \frac{\d p^0}{\d s_1}\right) - s_3^3R^2\frac{\d^2 p^0}{\d s_1^2} \right )\, ds_3
\\
&=\frac{2\pi R}{16 \rho_0 \nu} \left(2\frac{\d}{\d s_1}\left(R^2 \frac{\d p^0}{\d s_1}\right) - R^2\frac{\d^2 p^0}{\d s_1^2} \right).
\end{align*}
On the other hand,
\begin{align*}
\int_{\d\omega}\bphi^1\cdot \mathbf{n}= \int_0^{2\pi} \pfrac{R}{t}\, ds_2 = 2\pi\pfrac{R}{t}.
\end{align*}
Therefore, the compatibility condition (\ref{cond_comp_U1}) is equivalent to
\begin{equation} \label{condÐcomp-U1-b}
 \frac{R}{16 \rho_0 \nu} \left(2\frac{\d}{\d s_1}\left(R^2 \frac{\d p^0}{\d s_1}\right) - R^2\frac{\d^2 p^0}{\d s_1^2} \right) = \pfrac{R}{t},
\end{equation}
and it is easy to deduce from (\ref{p^0}) that (\ref{condÐcomp-U1-b}) is verified. Then, we can conclude that there exists a solution $(\mathbf{U^1}, p^2)$ of (\ref{problema_U1}) such that $\mathbf{U^1}$ is unique and $p^2$ is unique up to a function depending on $t$ and $s_1$. 
This solution can be computed explicitly (see appendix \ref{apend.U1} for the details), and it is 
\begin{align*}
\mathbf{U^1}&= \frac{R}{16 \rho_0 \nu}\left( 2\frac{\d}{\d s_1}\left(R^2 \frac{\d p^0}{\d s_1}\right) - (z_2^2+z_3^2)R^2\frac{\d^2 p^0}{\d s_1^2} \right)\left(z_2,z_3\right), \\
p^2&=- \frac{R^2}{4}\pfrac{^2p^0}{s_1^2}\left(z_2^2 + z_3^2 \right) + p_0^2(t,s_1), 
\end{align*}
where the unknown function $p_0^2$ will be determined later (see (\ref{eq-p02})).

\item Grouping the terms multiplied by $\var^0=1$ in (\ref{NavierStokes_asymp}), we find 
\begin{align}\nonumber
&\pfrac{(u_k^0 v_{ki})}{t} - \frac{s_3}{R}\pfrac{R}{t}\pfrac{(u_k^0 v_{ki})}{s_3} + \left( -\frac{\sin s_2}{R s_3}u_2^1 + \frac{\cos s_2}{R s_3}u_3^1\right)\pfrac{(u_k^0v_{ki})}{s_2} 
\\ \nonumber
&\quad {} + \left(\frac{\cos s_2}{R}u_2^1 + \frac{\sin s_2}{R}u_3^1\right)\pfrac{(u_k^0 v_{ki})}{s_3} + \pfrac{(u_k^0 v_{ki})}{s_1}u_1^0 - \tau\pfrac{(u_k^0 v_{ki})}{s_2}u_1^0  
\\ \nonumber
&\quad {} + \left(-\frac{\sin s_2}{R s_3}u_2^0 + \frac{\cos s_2}{R s_3}u_3^0\right)\pfrac{(u_k^1 v_{ki})}{s_2} -\frac{s_3}{R}\pfrac{R}{s_1}\pfrac{(u_k^0v_{ki})}{s_3}u_1^0 \\ \nonumber
&\quad {} +\left( \frac{\cos s_2}{R}u_2^0 + \frac{\sin s_2}{R}u_3^0 \right)\pfrac{(u_k^1 v_{ki})}{s_3} 
\\ \nonumber 
&\quad {} -\nu\left(\kappa^2\cos s_2\sin s_2 \pfrac{(u_k^0 v_{ki})}{s_2} - \kappa^2s_3\cos^2s_2\pfrac{(u_k^0 v_{ki})}{s_3} + \pfrac{^2(u_k^0 v_{ki})}{s_1^2} \right.
\\ \nonumber
&\quad {} - \tau'\pfrac{(u_k^0v_{ki})}{s_2} -\frac{s_3}{R}\left(\pfrac{^2R}{s_1^2} - \frac{2}{R}\left(\pfrac{R}{s_1} \right)^2\right)\pfrac{(u_k^0v_{ki})}{s_3} -2\tau\pfrac{^2(u_k^0 v_{ki})}{s_1\d s_2} 
\\ \nonumber 
&\quad {} -2\frac{s_3}{R}\pfrac{R}{s_1} \pfrac{^2(u_k^0v_{ki})}{s_1\d s_3} 
+ 2\tau\frac{s_3}{R}\pfrac{R}{s_1}\pfrac{^2 (u_k^0 v_{ki})}{s_3\d s_2} + \tau^2 \pfrac{^2(u_k^0 v_{ki})}{s_2^2}
\\&\nonumber
\quad {} + \frac{s_3^2}{R^2}\left(\pfrac{R}{s_1} \right)^2\pfrac{^2(u_k^0 v_{ki})}{s_3^2} + \kappa\frac{\sin s_2}{R s_3}\pfrac{(u_k^1 v_{ki})}{s_2} - \kappa\frac{\cos s_2}{R}\pfrac{(u_k^1v_{ki})}{s_3} 
\\&\left. \label{et18}
\quad {} + \frac{1}{(Rs_3)^2}\pfrac{^2(u_k^2 v_{ki})}{s_2^2} + \frac{1}{R^2s_3}\pfrac{(u_k^2 v_{ki})}{s_3} + \frac{1}{R^2}\pfrac{^2(u_k^2 v_{ki})}{s_3^2}   \right)= DP^0_kv_{ki} + b^0_{0k}v_{ki},
\end{align}
where,
\begin{align*}
DP_1^0&=-\frac{1}{\rho_0}\left(\kappa^2s_3^2R^2\cos^2 s_2 \pfrac{p^0}{s_1} + \kappa s_3R\cos s_2 \pfrac{p^1}{s_1} - \tau\kappa s_3R \cos s_2 \pfrac{p^1}{s_2}  \right.
\\
& \left. \qquad {} -\kappa s_3^2 \cos s_2 \pfrac{R}{s_1}\pfrac{p^1}{s_3} + \pfrac{p^2}{s_1} 
- \tau \pfrac{p^2}{s_2} - \frac{s_3}{R}\pfrac{R}{s_1}\pfrac{p^2}{s_3} \right),
\\
DP_2^0&=-\frac{1}{\rho_0} \left(-\frac{\sin s_2}{R s_3} \pfrac{p^3}{s_2} + \frac{\cos s_2}{R }\pfrac{p^3}{s_3} \right),
\\
DP_3^0&=-\frac{1}{\rho_0} \left(\frac{\cos s_2}{R s_3} \pfrac{p^3}{s_2} + \frac{\sin s_2}{R }\pfrac{p^3}{s_3} \right).
\end{align*}

Now, we use the expressions obtained in steps $(i)$ and $(ii)$ (see (\ref{u_1^0})--(\ref{edp_p1}) and (\ref{et-p2})), and we replace them into equation (\ref{et18}). Since $\{ \bbv_1,\bbv_2,\bbv_3 \}$ is an orthonormal basis, we obtain three equations by grouping the terms multiplied by each vector in (\ref{et18}). Therefore, from the terms multiplied by $\bbv_1$ in (\ref{et18}), we obtain
 \begin{align*}
 &\frac{1}{(Rs_3)^2}\frac{\d^2 u_1^2}{\d s_2^2} + \frac{1}{R^2s_3}\frac{\d u_1^2}{\d s_3} + \frac{1}{R^2}\frac{\d^2u_1^2}{\d s_3^2} \\
 &\quad = \left( \frac{R^2}{4\rho_0\nu^2} \pfrac{^2 p^0}{t\d s_1} - \frac{R^4}{16 \rho^2\nu^3}\pfrac{p^0}{s_1}\pfrac{^2p^0}{s_1^2} -\frac{R^2}{2 \rho_0\nu}\pfrac{^3p^0}{s_1^3} +\frac{7 \kappa^2 R^2}{16\rho_0 \nu}\pfrac{p^0}{s_1} \right)s_3^2
 \\
 & \qquad {} +\frac{R^4}{32\rho_0^2\nu^3}\pfrac{p^0}{s_1}\pfrac{^2 p^0}{s_1^2} s_3^4 -\frac{1}{4\rho_0\nu^2} \pfrac{}{t}\left( R^2 \pfrac{p^0}{s_1}\right) + \frac{R^2}{16 \rho_0^2 \nu^3}\pfrac{p^0}{s_1}\pfrac{}{s_1}\left(R^2 \pfrac{p^0}{s_1} \right) \\
 &\qquad {} + \frac{1}{4 \rho_0 \nu} \pfrac{^2}{s_1^2}\left( R^2 \pfrac{p^0}{s_1} \right) - \frac{7\kappa^2 R^2}{16 \rho_0\nu}\pfrac{p^0}{s_1} +\frac{1}{\rho_0\nu}\pfrac{p_0^2}{s_1} \\
 &\qquad {} +\frac{3\kappa R}{2\rho_0\nu}\pfrac{p^1}{s_1}s_3 \cos s_2 + \frac{15 \kappa^2 R^2}{8 \rho_0\nu}\pfrac{p^0}{s_1}s_3^2 \cos^2 s_2  -\frac{b_{01}}{\nu}.
 \end{align*}
 
This equation, together with the boundary condition at $s_3=1$ (see (\ref{bc_terms})), and using the change of variable (\ref{def_z}), shows that $u_1^2$ is the unique solution of the problem 
	\begin{equation}\label{pb-u12}
	\left \{ \begin{array}{l}
	\displaystyle\Delta_\mathbf{z}u_1^2 
	= \left( \frac{R^4}{4\rho_0\nu^2} \pfrac{^2 p^0}{t\d s_1} - \frac{R^6}{16 \rho^2\nu^3}\pfrac{p^0}{s_1}\pfrac{^2p^0}{s_1^2} -\frac{R^4}{2 \rho_0\nu}\pfrac{^3p^0}{s_1^3} +\frac{7 \kappa^2 R^4}{16\rho_0 \nu}\pfrac{p^0}{s_1} \right)(z_2^2 + z_3^2) \\[12pt]
\displaystyle \qquad {} +\frac{R^6}{32\rho_0^2\nu^3}\pfrac{p^0}{s_1}\pfrac{^2 p^0}{s_1^2} (z_2^2 + z_3^2)^2 -\frac{R^2}{4\rho_0\nu^2} \pfrac{}{t}\left( R^2 \pfrac{p^0}{s_1}\right)  \\[12pt]
\displaystyle \qquad {} + \frac{R^4}{16 \rho_0^2 \nu^3}\pfrac{p^0}{s_1}\pfrac{}{s_1}\left(R^2 \pfrac{p^0}{s_1} \right) 
+ \frac{R^2}{4 \rho_0 \nu} \pfrac{^2}{s_1^2}\left( R^2 \pfrac{p^0}{s_1} \right)  \\[12pt]
\displaystyle \qquad {} - \frac{7\kappa^2 R^4}{16 \rho_0\nu}\pfrac{p^0}{s_1} +\frac{R^2}{\rho_0\nu}\pfrac{p_0^2}{s_1} 
+ \frac{3\kappa R^3}{2\rho_0\nu}\pfrac{p^1}{s_1}z_2 + \frac{15 \kappa^2 R^4}{8 \rho_0\nu}\pfrac{p^0}{s_1}z_2^2  - R^2\frac{b_{01}}{\nu}
\quad \textrm{in} \ \omega, 
\\[12pt] \displaystyle u_1^2=0 \quad \textrm{on} \ \d \omega.
	\end{array} \right .
	\end{equation}

The solution of (\ref{pb-u12}) must be polynomial on $z_2$ and $z_3$ and, by inspection in this kind of functions, we can find that $u_1^2$ is 
\begin{align} \label{eq-u12-z2-z3}
 u_1^2&= \frac{R^2}{16} \left(\frac{R^2}{4 \rho_0 \nu^2} \frac{\d^2 p^0}{\d t\d s_1 } -\frac{R^4}{16\rho_0^2\nu^3}\frac{\d p^0}{\d s_1}\frac{\d^2p^0}{\d s_1^2}
- \frac{R^2}{2 \rho_0 \nu}\frac{\d^3 p^0}{\d s_1 ^3} 
 + \frac{11 \kappa^2 R^2}{8 \rho_0 \nu}\frac{\d p^0}{\d s_1}\right)((z_2^2+z_3^2)^2 - 1) 
\nonumber \\
& \qquad {} + \frac{R^2}{4} \left(-\frac1{4\rho_0 \nu^2}\frac{\d}{\d t}\left(R^2 \frac{\d p^0}{\d s_1}  \right) \right.
+ \frac{R^2}{16 \rho_0^2 \nu^3}\frac{\d p^0}{\d s_1} \frac{\d}{\d s_1} \left( R^2 \frac{\d p^0}{\d s_1}  \right) 
+ \frac1{4 \rho_0 \nu} \frac{\d^2}{\d s_1^2} \left(R^2 \frac{\d p^0}{\d s_1}\right) \nonumber
 \\ \displaystyle &\qquad \nonumber\left.
 {} -\frac{7 \kappa^2 R^2}{16 \rho_0 \nu}\frac{\d p^0}{\d s_1} +\frac1{\rho_0 \nu} \frac{\d p^2_0}{\d s_1} - \frac{b_{01}}{\nu}\right)(z_2^2+z_3^2-1 ) 
    + \frac{R^6}{1152 \rho_0^2 \nu^3} \frac{\d p^0}{\d s_1}\frac{\d^2 p^0}{\d s_1^2}((z_2^2+z_3^2)^3-1)
    \\& \displaystyle \qquad 
{} + \frac{3\kappa R^3 }{16 \rho_0 \nu}\frac{\d p^1}{\d s_1} (z_2^2+z_3^2-1)z_2 
+ \frac{5 \kappa^2 R^4}{64 \rho_0 \nu} \frac{\d p^0}{\d s_1} (z_2^2+z_3^2-1)(z_2^2-z_3^2),
 \end{align}
that, using the change of variable (\ref{def_z}), can be written
\begin{align} \label{eq-u12-s3}
 u_1^2&= \frac{R^2}{16} \left(\frac{R^2}{4 \rho_0 \nu^2} \frac{\d^2 p^0}{\d t\d s_1 } -\frac{R^4}{16\rho_0^2\nu^3}\frac{\d p^0}{\d s_1}\frac{\d^2p^0}{\d s_1^2}
- \frac{R^2}{2 \rho_0 \nu}\frac{\d^3 p^0}{\d s_1 ^3} 
 + \frac{11 \kappa^2 R^2}{8 \rho_0 \nu}\frac{\d p^0}{\d s_1}\right)(s_3^4 - 1) 
\nonumber \\
& \qquad {} + \frac{R^2}{4} \left(-\frac1{4\rho_0 \nu^2}\frac{\d}{\d t}\left(R^2 \frac{\d p^0}{\d s_1}  \right) \right.
+ \frac{R^2}{16 \rho_0^2 \nu^3}\frac{\d p^0}{\d s_1} \frac{\d}{\d s_1} \left( R^2 \frac{\d p^0}{\d s_1}  \right) 
+ \frac1{4 \rho_0 \nu} \frac{\d^2}{\d s_1^2} \left(R^2 \frac{\d p^0}{\d s_1}\right) \nonumber
 \\ \displaystyle &\qquad \nonumber\left.
 {} -\frac{7 \kappa^2 R^2}{16 \rho_0 \nu}\frac{\d p^0}{\d s_1} +\frac1{\rho_0 \nu} \frac{\d p^2_0}{\d s_1} - \frac{b_{01}}{\nu}\right)(s_3^2-1 ) 
    + \frac{R^6}{1152 \rho_0^2 \nu^3} \frac{\d p^0}{\d s_1}\frac{\d^2 p^0}{\d s_1^2}(s_3^6-1)
    \\& \displaystyle \qquad 
{} + \frac{3\kappa R^3 }{16 \rho_0 \nu}\frac{\d p^1}{\d s_1} (s_3^3- s_3)\cos s_2 
+ \frac{5 \kappa^2 R^4}{64 \rho_0 \nu} \frac{\d p^0}{\d s_1} (s_3^4- s_3^2)\cos(2 s_2).
 \end{align}
 
Using now (\ref{relaciones_flow}) for $k=2$, we have that 
\begin{align} \label{integ-u12-s3}
	\frac{\d}{\d s_1} \left(R^2\int_{0}^{2\pi} \int_0^1 s_3 u_1^2\,  ds_3 ds_2\right)= 0,
	\end{align}
and, if we substitute the expression of $u_1^2$ given by (\ref{eq-u12-s3}) in (\ref{integ-u12-s3}), 
we obtain that $p_0^2$ is the solution of the problem
\begin{align}
	\displaystyle \frac{\d}{\d s_1} \left(R^4 \frac{\d p_0^2}{\d s_1}\right) &= \frac{\d}{\d s_1} \left( - \frac{3 R^8}{64 \rho_0 \nu^2}\frac{\d p^0}{\d s_1}\frac{\d^2 p^0}{\d s_1^2}
	- \frac{R^6}{12} \frac{\d^3 p^0}{\d s_1^3} \right. - \frac{\kappa^2 R^6}{48} \frac{\d p^0}{\d s_1} 
	+\frac{R^5}{2\nu} \frac{\d R}{\d t} \frac{\d p^0}{\d s_1} \nonumber
	\\ \displaystyle  &\qquad 
	- \frac{R^7}{8 \rho_0 \nu^2} \frac{\d R}{\d s_1} \left(\frac{\d p^0}{\d s_1}\right)^2	-\frac{R^4}{2}\left(\frac{\d R}{\d s_1}\right)^2\frac{\d p^0}{\d s_1} 
	-\frac{R^5}{2}\frac{\d^2 R}{\d s_1^2}\frac{\d p^0}{\d s_1} \nonumber
	\\ \displaystyle& \left. \qquad  \label{eq-p02}
	- R^5 \frac{\d R}{\d s_1}\frac{\d^2 p^0}{\d s_1^2}  +  \frac{R^6}{6\nu} \frac{\d^2p^0}{\d t\d s_1}   + R^4 \rho_0 b_{01} \right),
	\end{align}
with the adequate initial and boundary conditions. With this equation, we have completed the description of $p^2$ (see (\ref{et-p2})--(\ref{eq-dif-p02})).

Identifying now the terms multiplied by $\bbv_2$ in (\ref{et18}), we obtain 
 \begin{align*}
  &\frac{1}{(Rs_3)^2}\frac{\d^2 u_2^2}{\d s_2^2} + \frac{1}{R^2s_3}\frac{\d u_2^2}{\d s_3} + \frac{1}{R^2}\frac{\d^2u_2^2}{\d s_3^2} 
= \frac{\kappa R^4}{16 \rho_0^2 \nu^3} \left(\pfrac{p^0}{s_1} \right)^2 \left(s_3^4 +1\right)  \\
&\quad {} + \frac{\kappa'R^2}{4\rho_0\nu}\pfrac{p^0}{s_1} + \frac{5\kappa}{8\rho_0\nu}\pfrac{}{s_1}\left( R^2\pfrac{p^0}{s_1}\right) 
    \\
    &\quad {} + \left(  -\frac{\kappa R^4}{8 \rho_0^2\nu^3}\left(\pfrac{p^0}{s_1} \right)^2 - \frac{\kappa}{2\rho_0\nu}\pfrac{}{s_1}\left(R^2 \pfrac{p^0}{s_1}\right) -\frac{\kappa'R^2}{4\rho_0\nu}\pfrac{p^0}{s_1}  + \frac{\kappa R}{\rho_0\nu}\pfrac{R}{s_1}\pfrac{p^0}{s_1} - \frac{\kappa R^2}{16\rho_0\nu} \pfrac{^2p^0}{s_1^2}   \right)s_3^2 
  \\
  & \quad {} - \frac{\kappa R^2 }{8\rho_0\nu}\pfrac{^2p^0}{s_1^2}s_3^2\cos^2 s_2 
+ \frac{1}{\rho_0\nu}\left(-\frac{\sin s_2}{R s_3 }\pfrac{p^3}{s_2} + \frac{\cos s_2}{R}\pfrac{p^3}{s_3}\right) -\frac{ b_{02}}{\nu},
 \end{align*}
 and if we identify the terms of (\ref{et18}) multiplied by $\bbv_3$, 
   \begin{align*}\nonumber
   &\frac{1}{(Rs_3)^2}\frac{\d^2 u_3^2}{\d s_2^2} + \frac{1}{R^2s_3}\frac{\d u_3^2}{\d s_3} + \frac{1}{R^2}\frac{\d^2u_3^2}{\d s_3^2}= 
- \frac{\kappa \sin s_2 \cos s_2}{16s_3\rho_0\nu}\left(2 \pfrac{}{s_1}\left(R^2 \pfrac{p^0}{s_1}\right)s_3 - R^2\pfrac{^2p^0}{s_1^2}s_3^3\right) 
   \\ \nonumber
   &\qquad+ \frac{\kappa \sin s_2 \cos s_2}{16\rho_0\nu}\left(2 \pfrac{}{s_1}\left(R^2 \pfrac{p^0}{s_1}\right) - 3R^2\pfrac{^2p^0}{s_1^2}s_3^2\right) -\frac{\kappa\tau R^2}{4\rho_0\nu}\pfrac{p^0}{s_1}(s_3^2 -1) 
   \\
   &\qquad + \frac{1}{\rho_0\nu}\left( \frac{\cos s_2}{R s_3}\pfrac{p^3}{s_2} + \frac{\sin s_2}{R}\pfrac{p^3}{s_2} \right) -\frac{b_{03}}{\nu}.
  \end{align*}

 Using the change of variable (\ref{def_z}) and doing some simplifications, we deduce that 
 $u_2^2$ and $u_2^3$ verify 
 \begin{align}\nonumber
 \Delta_{\mathbf{z}}u_2^2&=  \frac{\kappa R^6}{16 \rho^2 \nu^3} \left(\frac{\d p^0}{\d s_1}\right)^2 \left( (z_2^2 + z_3^2)^2 +1 \right)   + \frac{\kappa' R^4}{4 \rho_0 \nu}\frac{\d p^0}{\d s_1} + \frac{5 \kappa R^2}{8 \rho_0 \nu} \frac{\d}{\d s_1}\left(R^2 \frac{\d p^0}{\d s_1} \right)
 \\ \nonumber
 & \qquad + \left(- \frac{\kappa R^6}{8 \rho^2_0 \nu^3} \left(\frac{\d p^0}{\d s_1}  \right)^2 - \frac{9\kappa R^4}{16 \rho_0 \nu}\frac{\d^2 p^0}{\d s_1^2} - \frac{\kappa' R^4}{4\rho_0 \nu}\frac{\d p^0}{\d s_1} \right){(z_2^2 + z_3^2)} 
\\ \label{lap_u22}
&\qquad  -\frac{\kappa R^4 }{8 \rho_0 \nu} \frac{\d^2 p^0}{\d s_1^2}z_2^2 + \frac{R}{\rho_0\nu}\pfrac{p^3}{z_2} 
 - \frac{R^2}{\nu}b_{02}, 
 \end{align}
 \begin{align} \label{lap_u23}
  \Delta_{\mathbf{z}}u_3^2&= -\frac{\kappa\tau R^4}{4 \rho_0 \nu} \frac{\d p^0}{\d s_1} (z_2^2 + z_3^2 -1) -\frac{2\kappa R^4}{16 \rho_0\nu}\frac{\d^2p^0}{\d s_1^2}z_2z_3
 + \frac{R}{\rho_0\nu}\pfrac{p^3}{z_3} 
  -\frac{R^2}{\nu}b_{03}.
  \end{align}
  
 Now, if we group the terms multiplied by $\var$ in (\ref{NavierStokes_asymp2}), we obtain 
 \begin{align*}
 &-\frac{\sin s_2}{R s_3}\pfrac{u_2^2}{s_2} + \frac{\cos s_2}{Rs_3}\pfrac{u_3^2}{s_2} + \frac{\cos s_2}{R}\pfrac{u_2^2}{s_3} + \frac{\sin s_2}{R}\pfrac{u_3^2}{s_3}= - \kappa s_3 R \cos s_2 \left( \pfrac{u_1^0}{s_1} 
- \frac{s_3}{R}\pfrac{R}{s_1}\pfrac{u_1^0}{s_3}\right) \\
 &\qquad -\pfrac{u_1^1}{s_1} + \kappa u_2^1 +\tau\pfrac{u^1_1}{s_2} + \frac{s_3}{R} \pfrac{R}{s_1}\pfrac{u_1^1}{s_3},
 \end{align*}
 and using in this equation the expressions obtained in steps $(i)$ and $(ii)$ (see (\ref{u_1^0})--(\ref{edp_p1})),
  \begin{align} \nonumber
  &-\frac{\sin s_2}{s_3}\pfrac{u_2^2}{s_2} + \frac{\cos s_2}{s_3}\pfrac{u_3^2}{s_2} + (\cos s_2)\pfrac{u_2^2}{s_3} + (\sin s_2)\pfrac{u_3^2}{s_3} \\ \nonumber 
  &\quad = - \kappa R^2 s_3 \cos s_2 \left(  \frac{1}{4\rho_0\nu}\pfrac{}{s_1}\left(R^2\pfrac{p^0}{s_1}\right)(s_3^2-1)  -\frac{R}{2\rho_0\nu} \pfrac{R}{s_1}\pfrac{p^0}{s_1}s_3^2   \right) 
  \\ \nonumber 
&\qquad {}   -\frac{3R}{16\rho_0\nu}\pfrac{}{s_1}\left(\kappa R^3\pfrac{p^0}{s_1}\right)s_3(s_3^2-1)\cos s_2 
  -\frac{R}{4\rho_0\nu}\pfrac{}{s_1}\left(R^2\pfrac{p^1}{s_1}\right)(s_3^2-1) 
  \\ \nonumber
  & \qquad+ \frac{\kappa R^2}{16\rho_0\nu}\left( 2\pfrac{}{s_1}\left(R^2\pfrac{p^0}{s_1} \right) - R^2 s_3^2 \pfrac{^2p^0}{s_1^2}\right)s_3\cos s_2 -\frac{3\kappa \tau R^4}{16\rho_0\nu}\pfrac{p^0}{s_1}s_3(s_3^2-1)\sin s_2 
   \\ \label{div_u22_u23}
    & \qquad +\frac{3 \kappa R^3}{16\rho_0\nu}\pfrac{R}{s_1}\frac{p^0}{s_1}(3s_3^2-1)s_3\cos s_2 + \frac{R^2}{2\rho_0\nu}\pfrac{R}{s_1}\pfrac{p^1}{s_1}s_3^2=:g.    
 \end{align} 
  Let	be $\mathbf{U}^2 = (u^2_2, u^2_3)$. Using the change of variable (\ref{def_z}), 
  we obtain from the equations (\ref{lap_u22}), (\ref{lap_u23}), (\ref{div_u22_u23}) and the boundary conditions (\ref{bc_terms}), that $(\mathbf{U}^2,p^3)$ solves the following problem, 
	\begin{equation*}	
		\left \{ \begin{array}{lcl}
		\displaystyle \Delta_\mathbf{z} \mathbf{U}^2 = \frac{R}{\rho_0 \nu}\nabla_\mathbf{z} p^3 + \mathbf{F}&\textrm{in} &\omega, 
		\\  \displaystyle \mathrm{div}\, \mathbf{U}^2 = g&\textrm{in} &\omega, 
		\\  \displaystyle 
		\mathbf{U}^2 = \mathbf{0} &\textrm{on} &\partial \omega, 
		\end{array} \right .
		\end{equation*}
where,
\begin{align*}
\mathbf{F}&:= \left( \frac{\kappa R^6}{16 \rho_0^2 \nu^3} \left(\frac{\d p^0}{\d s_1}\right)^2 \left( (z_2^2 + z_3^2)^2 +1 \right)   + \frac{\kappa' R^4}{4 \rho_0 \nu}\frac{\d p^0}{\d s_1} + \frac{5 R^2 \kappa}{8 \rho_0 \nu} \frac{\d}{\d s_1}\left(R^2 \frac{\d p^0}{\d s_1} \right) \right.
 \\ \nonumber
 & \qquad {} + \left(- \frac{\kappa R^6}{8 \rho^2_0 \nu^3} \left(\frac{\d p^0}{\d s_1}  \right)^2 - \frac{9R^4 \kappa}{16 \rho_0 \nu}\frac{\d^2 p^0}{\d s_1^2} - \frac{\kappa' R^4}{4\rho_0 \nu}\frac{\d p^0}{\d s_1} \right){(z_2^2 + z_3^2)} 
\\ 
&\qquad \left. {} -\frac{\kappa R^4 }{8 \rho_0 \nu} \frac{\d^2 p^0}{\d s_1^2}z_2^2 
 - \frac{R^2}{\nu}b_{02}, -\frac{\kappa\tau R^4}{4 \rho_0 \nu} \frac{\d p^0}{\d s_1} (z_2^2 + z_3^2 -1) -\frac{2R^4\kappa}{16 \rho_0\nu}\frac{\d^2p^0}{\d s_1^2}z_2z_3 
   -\frac{R^2}{\nu}b_{03}   \right),
\end{align*}
and $g$ is given by (\ref{funcion_g}) (from the definition of $g$ in (\ref{div_u22_u23}), and using 
(\ref{def_z}) after some simplification (see (\ref{funcion_g2})), we can obtain that $g$ is given by (\ref{funcion_g})).

Applying Theorem \ref{Teman}, this problem has a unique solution if the compatibility condition 
\begin{equation}\label{cond_compU2}
\int_\omega g\, dz_2dz_3= 0,
\end{equation}
is fulfilled, or equivalently, using the change of variable (\ref{def_z}), if
\begin{equation*}
\int_0^{2\pi}\int_0^1 s_3 g\, ds_3 ds_2= 0.
\end{equation*}

Taking into account the definition of g (see the right-hand side of (\ref{div_u22_u23})),
we obtain that
\begin{align*}
\int_0^{2\pi}\int_0^1 s_3 g\, ds_3 ds_2&= \frac{\pi R }{8 \rho_0\nu}\pfrac{}{s_1}\left(R^2 \pfrac{p^1}{s_1} \right) + \frac{\pi R^2}{4\rho_0\nu}\pfrac{R}{s_1}\pfrac{p^1}{s_1} \\
&= \frac{\pi R^2}{2\rho_0\nu}\pfrac{R}{s_1}\pfrac{p^1}{s_1} + 
\frac{\pi R^3 }{8 \rho_0\nu}\frac{\d^2 p^1}{\d s_1^2}.
\end{align*}
From (\ref{edp_p1}) we deduce that
\begin{align*}
4R^3\pfrac{R}{s_1}\pfrac{p^1}{s_1} + R^4\pfrac{^2p^1}{s_1^2}=0,
\end{align*}
so the compatibility condition (\ref{cond_compU2}) is verified and the problem (\ref{eqU2}) has uniqueness of solution. Furthermore, since $g$ and $\mathbf{F}$ are polynomial on $s_3$ (see (\ref{funcion_g})--(\ref{funcion_F})), the solution of (\ref{eqU2}) must be also polynomial on $s_3$ 
and we can compute it explicitly (see appendix \ref{apend.U2} for the details).
  
\end{enumerate}
\end{proof}

\section{Behavior of the wall of the pipe} \setcounter{equation}{0}

We need to close the equations of the model presented here (see theorem \ref{terminos_identificados}) with a law describing the behavior of the wall
of the pipe, that is, with a equation that allows us to determine $R$. There are different possibilities: a rigid wall, elastic or viscoelastic laws, etc.

The simplest case is when we consider that the wall of the pipe is rigid. We have assumed 
that $R(t, s_1)$ is given in (\ref{eq-phi-1-2})--(\ref{aplicacion}), so it is enough to suppose that 
\begin{displaymath}
\frac{\d R}{\d t} = 0
\end{displaymath}
to obtain a rigid wall. If we consider the steady case and a rigid wall of the pipe, our model reduces to the model obtained in \cite{Paloka}.

Other simple case is when we consider an algebraic elastic law (see \cite{Quarteroni}): 
\begin{equation}
\label{law}
p^0 - p_e = \frac{Eh_0}{R_0^2}(R-R_0)
\end{equation} 
where $E$ is the
Young modulus of the wall, $h_0$ its thickness, $R_0$ the radius of
the cross-section at rest, and $p_e$ is the external pressure. 

More complex (elastic or viscoelastic) laws can also be considered (see \cite{Quarteroni} again).

\section{Some numerical examples} \label{sec-4} \setcounter{equation}{0}

In this section we shall present some numerical examples in some representative cases 
in order to illustrate 
the behavior of the approximated solution obtained in the previous sections. 

We start plotting the main tangential velocity $u^0_1$ and its corrections $u^1_1$ and $u^2_1$. 
We observe in Figure \ref{u01plot} that  $u^0_1$ is a Poiseuille flow (other works as \cite{Hydon,Panasenko} have also shown this behavior).

\begin{figure}[H]
 \begin{center}
  %\makebox[\linewidth]
  {\includegraphics[width=.35\linewidth]{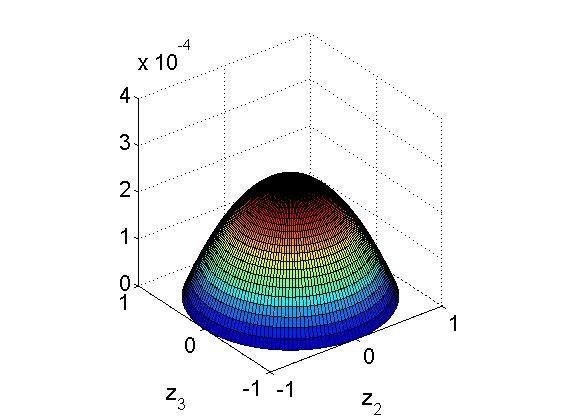}}
 \caption{Plot of $u_1^0 $ field.}\label{u01plot}
 \end{center}
\end{figure}

 In Figure \ref{u11plot} we can see that $u^1_1$ is a correction of $u^0_1$ that takes into account the curvature of the middle line (the fluid is faster in the side of the cross section of the pipe pointing to $\mathbf{N}$). 

\begin{figure}[H]
 \begin{center}
{\includegraphics[width=.35\linewidth]{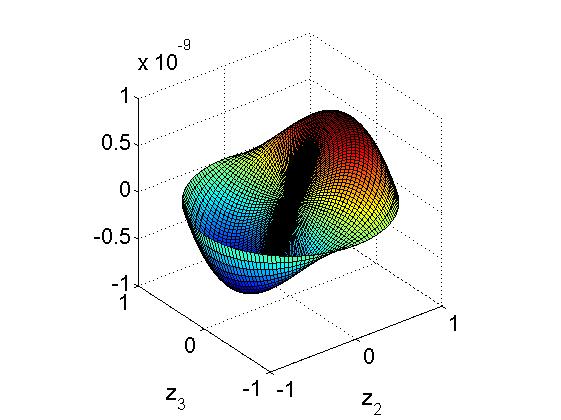}}
 \caption{Plot of $u_1^1 $ field.}\label{u11plot}
 \end{center}
\end{figure}

The correction of order two $u^2_1$, has a complex dependence on various terms (see (\ref{u_1^2})), but it is also similar to a Poiseuille flow (see Figure \ref{u21plot}). 

\begin{figure}[H]
 \begin{center}
  {\includegraphics[width=.35\linewidth]{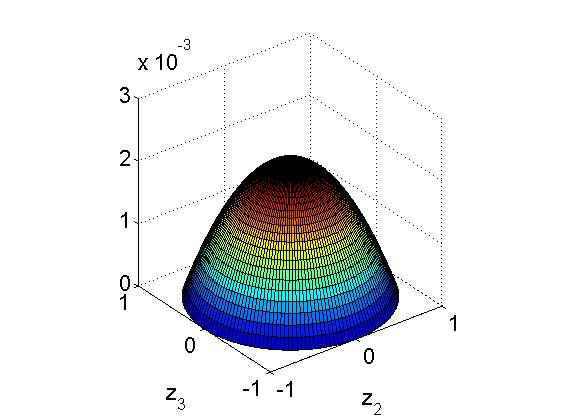}}
 \caption{Plot of $u_1^2 $ field.}\label{u21plot}
 \end{center}
\end{figure}

We have seen at (\ref{U^0}) that, at order zero, the transversal velocity is zero, so the tangential 
velocity is dominant. The first order correction, $\mathbf{U}^1 = (u^1_2, u^1_3)$, is related with the expansion and contraction of the pipe wall in radial direction. We can see in Figure \ref{U1plot} 
different cases depending on the value of $\frac{\partial p^0}{\partial s_1}$ (dp1), 
$\frac{\partial^2 p^0}{\partial s_1^2}$ (dp2) and $\frac{\partial r}{\partial s_1}$ (dr). 

\begin{figure}[h]
 \begin{center}
 {\includegraphics[width=.49\linewidth]{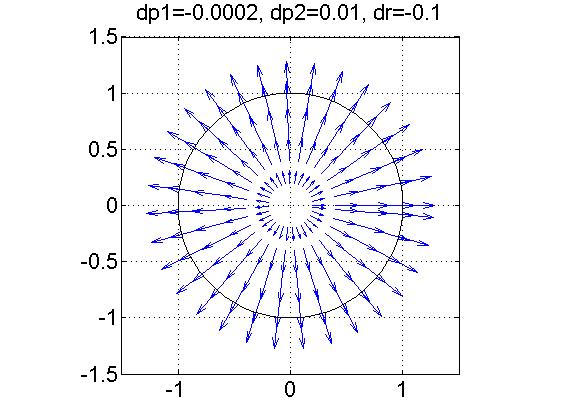}}
{\includegraphics[width=.49\linewidth]{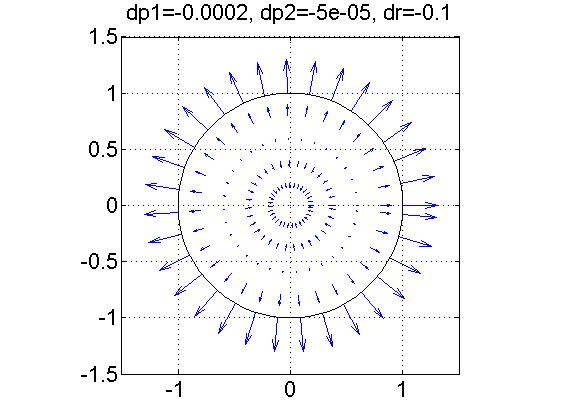}}
  {\includegraphics[width=.49\linewidth]{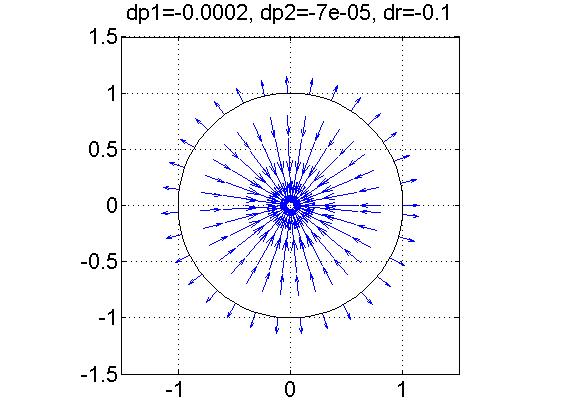}}
  {\includegraphics[width=.49\linewidth]{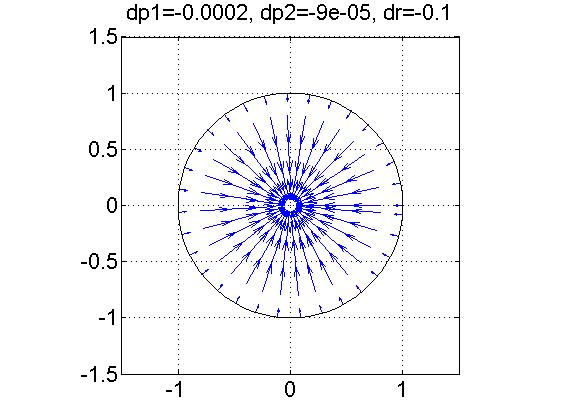}}
  \caption{Plot of $(u_2^1, u_3^1)$ field.}\label{U1plot}
 \end{center}
\end{figure}

The second order correction of transversal velocity, $\mathbf{U}^2=(u^2_2, u^2_3)$, is related 
with the recirculation of the fluid in the cross section of the pipe, as we can see in Figure \ref{U2plot}, 
where we show different cases depending on the curvature (k), its derivative (dk) and the torsion 
(tau) of the middle line of the pipe. 

\begin{figure}[h]
 \begin{center}
{\includegraphics[width=.49\linewidth]{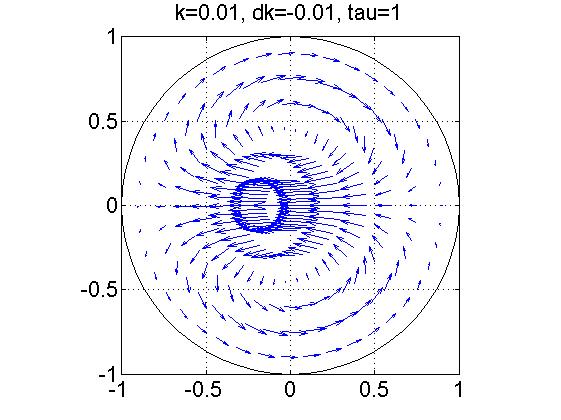}}
 {\includegraphics[width=.49\linewidth]{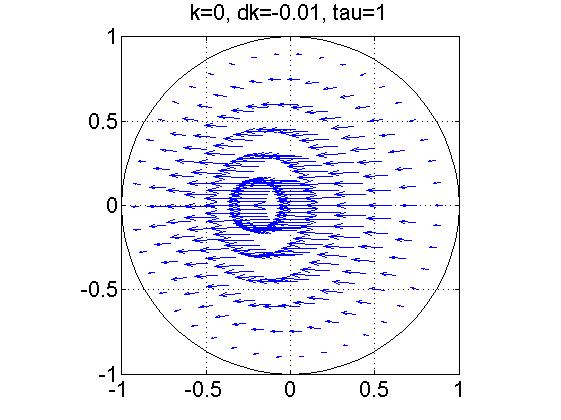}}
 {\includegraphics[width=.49\linewidth]{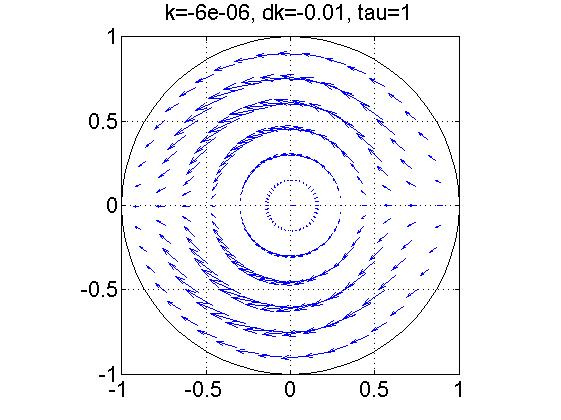}}
 {\includegraphics[width=.49\linewidth]{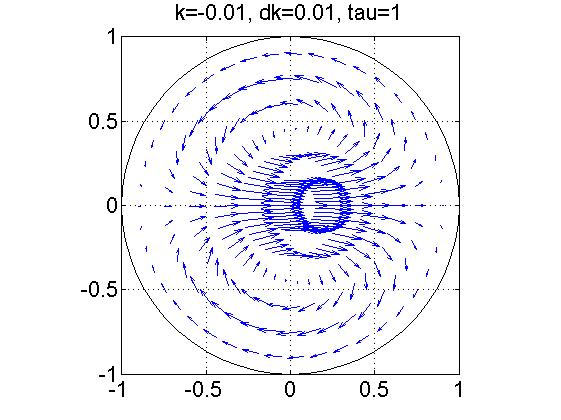}}
 \caption{Plot of $(u_2^2, u_3^2)$ field.}\label{U2plot}
 \end{center}
\end{figure}

\section{Conclusions} \label{sec-5} \setcounter{equation}{0}

A transient model for a newtonian fluid through a curved pipe with moving walls has been obtained. The asymptotic expansions have allowed us to find out the main components of velocity and their corrections. Furthermore, our model reduces to the obtained in \cite{Paloka}, when steady case and rigid walls are considered. Plots presented here (see figures \ref{u01plot}--\ref{U2plot}) compare very well with real patterns of fluid flow through a curved pipe and agree with the data available in the literature. A simple 
algebraic elastic law for the pipe wall has been considered in (\ref{law}), but other more general laws can be used.

%\clearpage
\section{Acknowledgements}

This research was partially supported by Ministerio de
Econom\'ia y Competitividad under grant MTM2012-36452-C02-01 with the participation of FEDER.

\clearpage

\begin{appendices}

\section{Computing $\mathbf{U}^1$ and $p^2$} \setcounter{equation}{0}\label{apend.U1}

Let us consider $(\mathbf{U^1}, p^2)$, the solution of problem (\ref{problema_U1}), 
\begin{equation*} \label{problema-U1-ap-a}
\left \{ \begin{array}{l}
\displaystyle\Delta_\mathbf{z}\mathbf{U^1}=\frac{R}{\nu\rho_0}\nabla_\mathbf{z}p^2 
\quad \textrm{in}\ \omega,
\\ \displaystyle \mathrm{div}_\mathbf{z}\mathbf{U^1}=\frac{R}{4 \rho_0 \nu}\left(\frac{\d}{\d s_1}\left(R^2 \frac{\d p^0}{\d s_1}\right) - (z_2^2+z_3^2)R^2\frac{\d^2 p^0}{\d s_1^2} \right)=:g^1 
\quad \textrm{in}\ \omega, 
\\ \displaystyle 
\mathbf{U^1} = \frac{\d R}{\d t}(\cos s_2, \sin s_2)=:\mathbf{\bphi^1} \quad \textrm{on} \ \d \omega,
\end{array} \right.
\end{equation*}
that we have seen in the proof of theorem \ref{terminos_identificados} that 
has a unique solution (in the case of $p^2$, up to an arbitrary function of $t$ and $s_1$). In order to compute $\mathbf{U^1}$ and $p^2$, we 
are going to consider some easier auxiliary problems. First, let us consider the problem 
		\begin{equation} \label{problema_phi}
	\left \{ \begin{array}{l}
			\displaystyle\Delta_\mathbf{z}\varphi=g^1 \ \textrm{in} \ \omega,
		\\ \displaystyle \frac{\d \varphi}{\d \mathbf{n}}= \bphi^1\cdot \mathbf{n} 
		= \frac{\d R}{\d t} \ \textrm{on} \ \d\omega,
		\end{array} \right.
		\end{equation}
		which has a unique solution (up to an arbitrary function of $t$ and $s_1$), since the compatibility condition (\ref{cond_comp_U1}) is verified. This problem can be written, 
using change of variable (\ref{def_z}), as follows 
\begin{equation} \label{eq-pb-phi-s3}
\left \{ \begin{array}{l}
\displaystyle \frac{1}{s_3^2} \pfrac{^2\varphi}{s_2^2}+ \frac{1}{s_3}\pfrac{\varphi}{s_3} + \pfrac{^2\varphi}{s_3^2} = \frac{R}{4 \rho_0 \nu}\left(\frac{\d}{\d s_1}\left(R^2 \frac{\d p^0}{\d s_1}\right) - s_3^2R^2\frac{\d^2 p^0}{\d s_1^2} \right)\quad \textrm{in}\ \omega, \\[10pt]
\displaystyle \pfrac{\varphi}{s_3}=\pfrac{R}{t}\quad \textrm{on}\ \d \omega.
\end{array} \right.
\end{equation}	

If we look for a solution of the form  
\begin{align*}
\varphi=a(t,s_1) s_3^4+ b(t,s_1)s_3^2 +c(t,s_1), 
\end{align*}	
we can identify $a$ and $b$ substituting in (\ref{eq-pb-phi-s3}). 
Taking into account (\ref{p^0}) to verify that the boundary condition is fulfilled,  
we find that 
\begin{align} \label{eq-phi-s3}
\varphi(t, s_1, s_2, s_3)= \frac{s_3^2R}{16 \rho_0 \nu}\left( \frac{\d}{\d s_1}\left(R^2 \frac{\d p^0}{\d s_1}\right) - \frac{s_3^2R^2}{4}\frac{\d^2 p^0}{\d s_1^2} \right) + c(t,s_1) 
\end{align}
is the solution of (\ref{eq-pb-phi-s3}). 
Using again (\ref{def_z}), we obtain that 
\begin{align} \label{eq-phi-z2-z3}
\varphi(t, s_1, z_2, z_3)= \frac{(z_2^2+z_3^2)R}{16 \rho_0 \nu}\left( \frac{\d}{\d s_1}\left(R^2 \frac{\d p^0}{\d s_1}\right) - (z_2^2+z_3^2)\frac{R^2}{4}\frac{\d^2 p^0}{\d s_1^2} \right) + c(t,s_1).	
\end{align}

		Let be $\mathbf{V} = \mathbf{U^1}- \nabla_{\mathbf{z}}\varphi$, and ${\bchi} =(-\sin s_2, \cos s_2)$ the unit  tangent vector on $\d\omega$. Then, from (\ref{problema_U1}) and (\ref{problema_phi}), we obtain that $(\mathbf{V},p^2)$ satisfies the problem
\begin{equation} \label{problema V}
		\left \{ \begin{array}{l}
		\displaystyle\Delta_\mathbf{z}\mathbf{V}= \frac{R}{\nu \rho_0} \nabla_{\mathbf{z}} p^2- \Delta_{\mathbf{z}}(\nabla_{\mathbf{z}}\varphi)   \quad \textrm{in}\ \omega,
		\\ \displaystyle \mathrm{div}_\mathbf{z}\mathbf{V}=0 \quad \textrm{in}\ \omega, 
		\\ \displaystyle 
		\mathbf{V\cdot n} =\mathbf{U^1\cdot n} - \nabla_{\mathbf{z}}\varphi\cdot \mathbf{n}=0 \quad \textrm{on $\d\omega$},
		\\ \displaystyle
		\mathbf{V\cdot \bchi} =\mathbf{U^1\cdot \bchi} - \nabla_{\mathbf{z}}\varphi\cdot \mathbf{\bchi}= - \frac{\d \varphi}{\d \boldmath{\bchi}}\quad \textrm{on $\d\omega$}.
		\end{array} \right.
\end{equation}

Since on $\d \omega$ we have that 
$\mathbf{U^1\cdot \bchi}=0$, $\displaystyle \nabla_{\mathbf{z}}\varphi\cdot \mathbf{\bchi} = 
\frac{\d \varphi}{\d \boldmath{\bchi}} = \pfrac{\varphi}{s_2}=0$, 
and $\varphi$ verifies in $\omega$ that 
\begin{align*}
\nabla_\mathbf{z}\varphi &= \frac{2R}{16 \rho_0 \nu}\left( \frac{\d}{\d s_1}\left(R^2 \frac{\d p^0}{\d s_1}\right) - (z_2^2+z_3^2)\frac{R^2}{2}\frac{\d^2 p^0}{\d s_1^2} \right)\left(z_2,z_3\right), \\
\Delta_\mathbf{z}\left(\nabla_\mathbf{z}\varphi \right) &= -\frac{R^3}{2\rho_0 \nu} \pfrac{^2 p^0}{s_1^2} (z_2,z_3),
\end{align*}
we obtain that 
$(\mathbf{V}, p^2)$ satisfies the problem 
\begin{equation} \label{eq-V-U1}
		\left \{ \begin{array}{l}
		\displaystyle\Delta_\mathbf{z}\mathbf{V}= \frac{R}{\nu \rho_0} \nabla_{\mathbf{z}} p^2+ \frac{R^3}{2\rho_0 \nu} \pfrac{^2 p^0}{s_1^2} (z_2,z_3)  \quad \textrm{in}\ \omega,
		\\ \displaystyle \mathrm{div}_\mathbf{z}\mathbf{V}=0 \quad \textrm{in}\ \omega, 
		\\ \displaystyle 
		\mathbf{V} =\bcero \quad \textrm{on $\d\omega$},
		\end{array} \right.
\end{equation}

Then, by Theorem \ref{Teman}, problem (\ref{eq-V-U1}) has a unique solution (up to an 
arbitrary function of $t$ and $s_1$, in the case of $p^2$), which expression is
\begin{align} \nonumber
\mathbf{V}&=\bcero, 
\\ \label{p_0^2}
p^2&=- \frac{R^2}{4}\pfrac{^2p^0}{s_1^2}\left(z_2^2 + z_3^2 \right) + p_0^2(t,s_1),
\end{align}
where $p_0^2(t,s_1)=c(t,s_1)$ is a smooth function, which is determined in (\ref{eq-p02}). Finally,  since $\mathbf{U^1}= \bV + \nabla_{\mathbf{z}}\varphi$, we have that
\begin{align*}
\mathbf{U^1}= \frac{R}{16 \rho_0 \nu}\left( 2\frac{\d}{\d s_1}\left(R^2 \frac{\d p^0}{\d s_1}\right) - (z_2^2+z_3^2)R^2\frac{\d^2 p^0}{\d s_1^2} \right)\left(z_2,z_3\right).
\end{align*}

\section{Computing $\mathbf{U}^2$ and $p^3$} \setcounter{equation}{0}\label{apend.U2}

Let us consider $(\mathbf{U}^2, p^3)$, solution of the problem 
	\begin{equation} \label{eq-U2-apendice-B}
		\left \{ \begin{array}{lcl}
		\displaystyle \Delta_\mathbf{z} \mathbf{U}^2 = \frac{R}{\rho_0 \nu}\nabla_\mathbf{z} p^3 + \mathbf{F}&\textrm{in} &\omega, 
		\\  \displaystyle \mathrm{div}\, \mathbf{U}^2 = g&\textrm{in} &\omega, 
		\\  \displaystyle 
		\mathbf{U}^2 = \mathbf{0} &\textrm{on} &\partial \omega, 
		\end{array} \right .
		\end{equation}
where $\mathbf{F}$ and $g$ are given, respectively, by (\ref{funcion_F}) and (\ref{div_u22_u23}).

In order to compute $(\mathbf{U}^2, p^3)$, we shall consider a decomposition of this problem in some easier ones (as done to compute $(\mathbf{U}^1, p^2)$ in appendix \ref{apend.U1}). 

First, let us consider the problem,
\begin{equation} \label{problema_phi2}
	\left \{ \begin{array}{l}
			\displaystyle\Delta_\mathbf{z}\varphi=g \ \textrm{in} \ \omega,
		\\ \displaystyle \frac{\d \varphi}{\d \mathbf{n}}=0 \ \textrm{on} \ \d\omega,
		\end{array} \right.
		\end{equation}
		which has unique solution, since the compatibility condition (\ref{cond_compU2}) is verified. 
		
Simplifying from (\ref{div_u22_u23}), we have
\begin{align}\nonumber
g&= \left(  - \frac{\kappa R^4}{2 \rho_0 \nu}\pfrac{^2p^0}{s_1^2} - \frac{3 \kappa' R^4}{16\rho_0 \nu}\pfrac{p^0}{s_1} \right) s_3^3\cos s_2  -\frac{3\kappa\tau R^4}{16\rho_0\nu}\pfrac{p^0}{s_1}s_3^3\sin s_2 
\\ \nonumber
&\quad+ \left(  \frac{9\kappa R^3}{8 \rho_0 \nu}\pfrac{R}{s_1}\pfrac{p^0}{s_1} + \frac{9 \kappa R^4}{16 \rho_0 \nu} \pfrac{^2p^0}{s_1^2} + \frac{3 \kappa' R^4}{16 \rho_0 \nu} \pfrac{p^0}{s_1}             \right)s_3\cos s_2
\\ \label{funcion_g2}
& \quad + \frac{3\kappa\tau R^4}{16\rho_0\nu}\pfrac{p^0}{s_1} s_3 \sin s_2 -\frac{R^3}{4\rho_0\nu}\pfrac{^2p^1}{s_1^2}s_3^2 + \frac{R}{4\rho_0\nu}\pfrac{}{s_1}\left( R^2 \pfrac{p^1}{s_1} \right), 
\end{align}
that, using (\ref{def_z}), can be written as in (\ref{funcion_g}). From (\ref{problema_phi2}), 
and using again (\ref{def_z}), we have that $\varphi$ is solution of 
\begin{equation}\label{et19}
\left \{ \begin{array}{l}
\displaystyle \frac{1}{s_3^2} \pfrac{^2\varphi}{s_2^2}+ \frac{1}{s_3}\pfrac{\varphi}{s_3} + \pfrac{^2\varphi}{s_3^2}=g \quad \textrm{in}\ \omega, \\[10pt]
\displaystyle \pfrac{\varphi}{s_3} = 0 \quad \textrm{on}\ \d \omega.
\end{array} \right.
\end{equation}

If we look for a solution of the form 
\begin{align}\nonumber
\varphi&=a(t,s_1) s_3^5\cos s_2 + b(t,s_1)s_3^5\sin s_2 + c(t,s_1)s_3^3\cos s_2 + d(t,s_1)s_3^3\sin s_2 
\\ \label{phi_U2}
& \quad
+ e(t,s_1)s_3 \cos s_2 
+ f(t,s_1)s_3\sin s_2 + j(t,s_1)s_3^4 + h(t,s_1)s_3^2 + i(t,s_1),
\end{align}	
where $a,b, c, d, e, f, g, h, i$ are smooth unknown functions, 
and we substitute in (\ref{et19}), we find that 
\begin{align} \nonumber
&\varphi(t,s_1,s_2,s_3)= \left( - \frac{R^4}{384\rho_0\nu}\left( 8\kappa\pfrac{^2p^0}{s_1^2} + 3\kappa'\pfrac{p^0}{s_1}    \right)\cos s_2 -\frac{\kappa\tau R^4}{128\rhonu}\dpc \sin s_2 \right)s_3^5 
\\ \nonumber
& \qquad  {} - \frac{R^3}{64\rhonu}\dpud   s_3^4 + \left(\frac{3 R^3}{128 \rhonu} \left( 6\kappa\dR\dpc + 3\kappa R\dpcd + \kappa'R\dpc  \right)\cos s_2  \right.
\\\nonumber
&\qquad \left. {} + \frac{3\kappa\tau R^4}{128 \rhonu}\dpc\sin s_2   \right) s_3^3 + \frac{R}{16 \rhonu}\pfrac{}{s_1}\left(R^2 \dpu\right)s_3^2
\\\nonumber
&\qquad {} + \left(   \left( \frac{5 R^4}{384 \rhonu}\left( 8\kappa \dpcd + 3\kappa'\dpc  \right) - \frac{9 R^3}{128\rhonu}\left(6\kappa\dR\dpc + \kappa'R\dpc \right. \right. \right.
\\
& \qquad \left. \left. \left. {} + 3\kappa R \dpcd\right)   \right)\cos s_2  -\frac{4 \kappa\tau R^4}{128 \rhonu} \dpc\sin s_2 \right) s_3 + i(t,s_1), \label{phi_2}
\end{align} 
where equation (\ref{edp_p1}) has been used to guarantee that the boundary condition in (\ref{et19}) 
is verified. Coming back to the local cartesian coordinates, we can write (\ref{phi_2}) 
\begin{align} \nonumber
&\varphi(t,s_1,z_2,z_3) 
= \left( - \frac{R^4}{384\rho_0\nu}\left( 8\kappa\pfrac{^2p^0}{s_1^2} 
+ 3\kappa'\pfrac{p^0}{s_1}   \right)z_2 -\frac{\kappa\tau R^4}{128\rhonu}\dpc z_3 \right.
\\ \nonumber 
&\qquad \left. {} - \frac{R^3}{64\rhonu}\dpud \right) (z_2^2 + z_3^2)^2 
+ \left(\frac{3 R^3}{128 \rhonu} \left( 6\kappa\dR\dpc + 3\kappa R\dpcd + \kappa'R\dpc  \right)z_2 
\right. \\ \nonumber
& \qquad  \left. {}  + \frac{3\kappa\tau R^4}{128 \rhonu}\dpc z_3   
+ \frac{R}{16 \rhonu}\pfrac{}{s_1}\left(R^2 \dpu \right) \right ) (z_2^2 + z_3^2)
\\ \nonumber
&\qquad {} + \left( \frac{5 R^4}{384 \rhonu}\left( 8\kappa \dpcd + 3\kappa'\dpc  \right) - \frac{9 R^3}{128\rhonu}\left(6\kappa\dR\dpc + \kappa'R\dpc + 3\kappa R \dpcd\right)   \right)z_2 
\\
& \qquad   {}   -\frac{4 \kappa\tau R^4}{128 \rhonu} \dpc z_3  + i(t,s_1). \label{phi_2-b}
\end{align} 

Now, let us consider $\mathbf{V} = \mathbf{U^2}- \nabla_{\mathbf{z}}\varphi$, and ${\bchi}=(-\sin s_2, \cos s_2)$ the unit  tangent vector on $\d\omega$. Then, from (\ref{eq-U2-apendice-B}) and (\ref{problema_phi2}), we obtain that $(\mathbf{V},p^3)$ satisfies the problem
\begin{equation} \label{problema V2}
		\left \{ \begin{array}{l}
		\displaystyle\Delta_\mathbf{z}\mathbf{V}= \frac{R}{\rho_0 \nu} \nabla_{\mathbf{z}} p^3 + \mathbf{F}- \Delta_{\mathbf{z}}(\nabla_{\mathbf{z}}\varphi)\quad \textrm{in}\ \omega,
		\\ \displaystyle \mathrm{div}_\mathbf{z}\mathbf{V}=0\quad \textrm{in}\ \omega, 
		\\ \displaystyle 
		\mathbf{V\cdot n} =\mathbf{U^2\cdot n} - \nabla_{\mathbf{z}}\varphi\cdot \mathbf{n}=0 \quad \textrm{on $\d\omega$},
		\\ \displaystyle
		\mathbf{V\cdot \bchi} =\mathbf{U^2\cdot \bchi} - \nabla_{\mathbf{z}}\varphi\cdot \mathbf{\bchi}= - \frac{\d \varphi}{\d \boldmath{\bchi}}\quad \textrm{on $\d\omega$}
		\end{array} \right.
\end{equation}

Let us consider an arbitrary smooth function $\psi$ and let be $\bW=\bV- \left( \frac{\d \psi}{\d z_3}, -\frac{\d \psi}{\d z_2} \right)$. Then $\bW$ is solution of the problem 
		\begin{equation}\label{problema_W}
		\left \{ \begin{array}{l}
			\displaystyle\Delta_\mathbf{z}\mathbf{\bW}= \frac{R}{\rho_0 \nu} \nabla_{\mathbf{z}} p^3 + \mathbf{F} - \Delta_{\mathbf{z}}(\nabla_{\mathbf{z}}\varphi)-\Delta_\mathbf{z}\left( \frac{\d \psi}{\d z_3}, -\frac{\d \psi}{\d z_2} \right)\quad \textrm{in}\ \omega,
		\\ \displaystyle \mathrm{div}_\mathbf{z}\mathbf{\bW}=0\quad \textrm{in}\ \omega,
		\\ \displaystyle 
		\mathbf{\bW\cdot n} =-\left( \frac{\d \psi}{\d z_3}, -\frac{\d \psi}{\d z_2} \right)\mathbf{n}  =-\nabla_{\mathbf{z}}\psi\cdot\bchi=- \frac{\d \psi}{\d \bchi} \quad \textrm{on $\d\omega$},
		\\ \displaystyle
		\mathbf{\bW\cdot \bchi} = - \frac{\d \varphi}{\d \boldmath{\bchi}}+ \left( \frac{\d \psi}{\d z_3}, -\frac{\d \psi}{\d z_2} \right) \cdot\bchi
		= - \frac{\d \varphi}{\d \boldmath{\bchi}}+ \frac{\d \psi}{\d \mathbf{n}}
		\quad \textrm{on $\d\omega$}.
		\end{array} \right.
		\end{equation}
		
If we are able to find a function $\psi$ such that 
\begin{equation} \label{eq-cond-psi}
\frac{\d \psi}{\d \bchi}=0 \quad \textrm{on}\ \d \omega, \qquad \frac{\d \psi}{\d \mathbf{n}}=  \frac{\d \varphi}{\d \boldmath{\bchi}} 
\quad \textrm{on}\ \d \omega,
\end{equation}
then problem (\ref{problema_W}) will be equivalent to 
		\begin{equation}\label{problema-W-2}
		\left \{ \begin{array}{l}
			\displaystyle\Delta_\mathbf{z}\mathbf{\bW}= \frac{R}{\rho_0 \nu} \nabla_{\mathbf{z}} p^3 + \mathbf{F} - \Delta_{\mathbf{z}}(\nabla_{\mathbf{z}}\varphi)-\Delta_\mathbf{z}\left( \frac{\d \psi}{\d z_3}, -\frac{\d \psi}{\d z_2} \right)\quad \textrm{in}\ \omega,
		\\ \displaystyle \mathrm{div}_\mathbf{z}\mathbf{\bW}=0\quad \textrm{in}\ \omega,
		\\ \displaystyle 
		\mathbf{\bW} =\mathbf{0} \quad \textrm{on $\d\omega$},
		\end{array} \right.
		\end{equation}
and we shall have, by Theorem \ref{Teman}, uniqueness of solution $(\bW, p^3)$ ($\bW$ is unique and $p^3$ is unique up to a function depending on $t$ and $s_1$). 

The next step is, following (\ref{eq-cond-psi}), finding a function $\psi$ such that
\begin{align}
\frac{\d \psi}{\d \bchi}= \frac{\d \psi}{\d s_2}=0 \quad \textrm{on}\ \d \omega, \label{eq-d-psi-s2} \\
\frac{\d \psi}{\d \mathbf{n}} = \frac{\d \psi}{\d s_3} =  \frac{\d \varphi}{\d \boldmath{\bchi}} = 
\frac{\d \varphi}{\d s_2} \quad \textrm{on}\ \d \omega.  \label{eq-d-psi-s3}
\end{align}
From (\ref{phi_2}) and (\ref{edp_p1}), we deduce that 
\begin{align*}
&\pfrac{\psi}{s_3}= \left( -\frac{4 R^4}{384\rhonu}\left(8\kappa \dpcd + 3\kappa'\dpc \right) 
\right. \\
&\qquad \left. {} + \frac{6 R^3}{128\rhonu} \left(6 \kappa \dR\dpc + \kappa R\dpc + 3\kappa R \dpcd \right)   \right) \sin s_2
- \frac{2\kappa\tau R^4}{128\rhonu} \dpc\cos s_2 
\quad \textrm{on}\ \d \omega,
\end{align*}
so, one possible election for $\psi$ is 
\begin{align*}
\psi(t, s_1, s_2, s_3)&= \frac{R^3}{384\rhonu} \left( \left(22\kappa R\dpcd + 6\kappa'R\dpc + 108\kappa\dR\dpc\right)\sin s_2 \right.
\\
& \qquad \left. -6 \kappa\tau R \dpc\cos s_2    \right)\frac{s_3(s_3^2-1)}{2},
\end{align*}
that, applying (\ref{def_z}), can be written 
\begin{align} \label{eq-psi-apendice-B}
\psi(t,s_1,z_2,z_3)= \left( \psi_2(t,s_1) z_2 + \psi_3(t,s_1) z_3 \right)\frac{z_2^2 + z_3^2 -1}{2}
\end{align}
where,
\begin{align}
\psi_3(t,s_1)&= \frac{R^3}{384\rhonu} \left(22\kappa R \dpcd + 6\kappa'R \dpc + 108\kappa\dR\dpc \right), \label{eq-psi-3} \\
\psi_2(t,s_1)&=- \frac{\kappa\tau R^4}{64\rhonu}  \dpc. \label{eq-psi-2}
\end{align}

Then, we have
\begin{align} \label{psi_U}
\left( \frac{\d \psi}{\d z_3}, -\frac{\d \psi}{\d z_2} \right)=\left( \frac{\psi_3}{2} \left( z_2^2 + 3 z_3^2 - 1\right) + \psi_2 z_2 z_3,  -\frac{\psi_2}{2} \left( 3z_2^2 +  z_3^2 - 1\right) - \psi_3 z_2 z_3  \right),
\end{align}
and 
\begin{align}
-\Delta_\mathbf{z}\left( \frac{\d \psi}{\d z_3}, -\frac{\d \psi}{\d z_2} \right)= 
(-4\psi_3, 4\psi_2). \label{eq-laplaciano-rot-psi}
\end{align}

Hence, from (\ref{problema-W-2}), $\bW$ is solution of the problem 
\begin{equation}\label{problema_W2}
		\left \{ \begin{array}{l}
		\displaystyle\Delta_\mathbf{z}\bW= \nabla_{\mathbf{z}} q^2 + \mathbf{F}   
		\quad \textrm{in}\ \omega,
		\\ \displaystyle \mathrm{div}_\mathbf{z}\bW=0 \quad \textrm{in}\ \omega, 
		\\ \displaystyle 
		\bW =\mathbf{0}\quad \textrm{on $\d\omega$},
		\end{array} \right.
		\end{equation} 
where 
\begin{align}\label{q2}
q^2 = \frac{R}{\rho_0 \nu} p^3 - g - 4\psi_3z_2 + 4\psi_2z_3 + q_0^2,
\end{align}
with $q_0^2 = q_0^2(t,s_1)$ an arbitrary smooth function of $t$ and $s_1$. 
From Theorem \ref{Teman}, we have the existence and uniqueness (up to an arbitrary function 
of $t$ and $s_1$, in the case of $q^2$) of $(\bW, q^2)$. 

Once we have computed $\mathbf{W}$ and $q^2$ (see below), 
we obtain $\mathbf{U}^2$ and $p^3$ in the following way,
\begin{align}
&\mathbf{U}^2 = \mathbf{W} + \nabla_{\mathbf{z}}\varphi+ \left( \pfrac{\psi}{z_3}, -\pfrac{\psi}{z_2}  \right), \\
&p^3 = \frac{\rho_0 \nu}{R} \left ( q^2 + g + 4\psi_3 z_2 - 4\psi_2 z_3 - q_0^2 \right ), 
\end{align}
where $\varphi$, $\psi$, $g$, $\psi_3$ and $\psi_2$ are given, respectively, by 
(\ref{phi_2-b}), (\ref{eq-psi-apendice-B}), (\ref{funcion_g}), (\ref{eq-psi-3}) and (\ref{eq-psi-2}).

Let us now compute $(\mathbf{W}, q^2)$. In order to make such computation, let us remark that $\mathbf{F}=(F_2, F_3)$ is polynomial in $z_2$ and $z_3$. 
Developing, from (\ref{funcion_F}), $F_2$ and $F_3$ in powers of $z_2$ and $z_3$, 
we obtain that 
\begin{align*}
F_2&=f_2^{00}+ f_2^{20}z_2^2 + f_2^{02}z_3^2 + f_2^{22}z_2^2z_3^2 + f_2^{40}z_2^4 + f_2^{04}z_3^4,
\\
F_3&=f_3^{00}+   f_3^{11}z_2z_3 + f_3^{20}z_2^2 + f_3^{02}z_3^2,
\end{align*}
where $f_\alpha^{mn}$ denotes the coefficient that multiplies $z_2^mz_3^n$ in $F_\alpha$. 
Since $\mathbf{F}$ is polynomial in $z_2$ and $z_3$, 
$\mathbf{W}=(W_2,W_3)$ and $q^2$ must be also polynomial in $z_2$ and $z_3$, 
so let us suppose that 
\begin{align} \nonumber
W_2&= \left(w_2^{00} + w_2^{10}z_2 + w_2^{01}z_3 + w_2^{20}z_2^2 + w_2^{11}z_2z_3 +w_2^{02}z_3^2 + w_2^{30}z_2^3 + w_2^{21}z_2^2z_3 + w_2^{12}z_2z_3^2 + w_2^{03}z_3^3 \right.
\\ \label{expr_W1}
& \left. \qquad {} + w_2^{40}z_2^4 + w_2^{31}z_2^3z_3 + w_2^{22}z_2^2z_3^2 + w_2^{13}z_2z_3^3 + w_2^{04}z_3^4\right)(z_2^2 + z_3^2 -1),
\\ \nonumber
W_3&= \left(w_3^{00} + w_3^{10}z_2 + w_3^{01}z_3 + w_3^{20}z_2^2 + w_3^{11}z_2z_3 +w_3^{02}z_3^2 + w_3^{30}z_2^3 + w_3^{21}z_2^2z_3 + w_3^{12}z_2z_3^2 + w_3^{03}z_3^3 \right.
\\ \label{expr_W2}
& \left. \qquad {} + w_3^{40}z_2^4 + w_3^{31}z_2^3z_3 + w_3^{22}z_2^2z_3^2 + w_3^{13}z_2z_3^3 + w_3^{04}z_3^4\right)(z_2^2 + z_3^2 -1), \\
\nonumber
q^2&= q^{00}+ q^{10}z_2 + q^{01}z_3 + q^{20}z_2^2 + q^{11}z_2z_3 + q^{02}z_3^2 + q^{30}z_2^2 + q^{21}z_2^2z_3 + q^{12}z_2z_3^2 + q^{03}z_3^3
 \\ \nonumber
&\qquad {} + q^{40}z_2^4 + q^{31}z_2^3z_3 + q^{22}z_2^2z_3^2 + q^{13}z_2z_3^3 
+ q^{04}z_3^4 + q^{50}z_2^5 + q^{41}z_2^4z_3 + q^{32}z_2^3z_3^2 
\\
&\qquad \label{expr_q}
{} + q^{23}z_2^2z_3^3 + q^{14}z_2z_3^4 + q^{05}z_3^5.
\end{align}

By substitution in (\ref{problema_W2}), we obtain a linear system with a unique solution, so 
$\mathbf{W}$ and $q^2$ are determined from the coefficients of the field $\mathbf{F}$ and 
our assumption about the form of $\mathbf{W}$ and $q^2$ was correct.

In this way we find that,
\begin{align} \label{W1}
\left\{\begin{aligned}[c]
w_2^{00}&= \frac{1}{192}\left(f_3^{11} - f_2^{04} \right) - \frac{1}{1152}f_2^{22} - \frac{1}{96}f_2^{02}, 
&  w_2^{04}&=\frac{7}{240}f_2^{04} - \frac{7}{2880}f_2^{22},
\\
 w_2^{02}&=\frac{5}{96}f_2^{02} + \frac{13}{960}f_2^{04} + \frac{31}{5760}f_2^{02} - \frac{5}{192}f_3^{11}, \quad & w_2^{11}&=-\frac{1}{24}f_3^{20},
\\
w_2^{20}&=\frac{1}{96}f_2^{02} + \frac{7}{960}f_2^{04} -\frac{11}{5760}f_2^{22} -\frac{1}{192}f_3^{11}, & w_2^{22}&=\frac{1}{480}f_2^{04} + \frac{37}{2880}f_2^{22},
\\
w_2^{40}&=\frac{1}{360}f_2^{22} - \frac{1}{480}f_2^{04}, & w_3^{00}&=-\frac{1}{96}f_3^{20},
\\
 w_3^{11}&=\frac{1}{480}f_2^{22} - \frac{1}{40}f_2^{04} - \frac{1}{24}f_2^{02} + \frac{1}{48}f_3^{11}, & w_3^{02}&=\frac{1}{96}f_3^{20},
\\ w_3^{13}&=- \frac{1}{80}f_2^{04} - \frac{1}{240} f_2^{22}, & w_3^{20}&=\frac{5}{96}f_3^{20},
\\
w_3^{31}&=\frac{1}{80}f_2^{04} - \frac{1}{60}f_2^{22},
\end{aligned}\right.
\end{align}
while 
\begin{align}\label{W2}
\left\{\begin{aligned}[c]
w_2^{01}&=w_2^{03}=w_2^{10}=w_2^{12}=w_2^{13}=w_2^{21}=w_2^{30}=w_2^{31}=0,
\\w_3^{01}&=w_3^{03}=w_3^{04}=w_3^{10}=w_3^{12}=w_3^{21}=w_3^{22}=w_3^{30}=w_3^{40}=0,
\end{aligned}\right.
\end{align}
and
\begin{align}\label{q1}
\left\{\begin{aligned}[c]
q^{10}&= \frac{1}{12}f_3^{11} - \frac{1}{6}f_2^{02} - \frac{1}{16}f_2^{04} - \frac{1}{96}f_2^{22} - f_2^{00}, & q^{03}&=\frac{1}{12}f_3^{20} - \frac{1}{3}f_3^{02},
\\
q^{12}&=\frac{3}{40}f_2^{22} - \frac{3}{20}f_2^{04} - \frac{1}{4}f_2^{02} - \frac{3}{8}f_3^{11}, & q^{01}&= -f_3^{00} - \frac{1}{6}f_3^{20},
\\
q^{14}&=- \frac{1}{16}f_2^{04} - \frac{5}{96}f_2^{22}, & q^{21}&=-\frac{1}{4}f_3^{20},
\\
q^{30}&=\frac{1}{12}f_2^{02} + \frac{1}{20}f_2^{04} - \frac{1}{3}f_2^{20} - \frac{1}{40}f_2^{22} - \frac{1}{24}f_3^{11}, & q^{32}&=\frac{1}{8}f_2^{04} - \frac{11}{48}f_2^{22},
\\
q^{50}&=\frac{11}{480}f_2^{22} - \frac{1}{80}f_2^{04} - \frac{1}{5}f_2^{40},
\end{aligned}\right.
\end{align}
while 
\begin{align} \label{q22}
q^{02}=q^{04}=q^{05}=q^{11}=q^{13}=q^{20}=q^{22}=q^{23}=q^{31}=q^{40}=q^{41}=0,
\end{align}
and $q^{00}$ is an arbitrary smooth function depending only on $t$ and on $s_1$.

\end{appendices}

%\section*{References}
\bibliographystyle{abbrv}%por orden alfab?tico
%\bibliographystyle{elsarticle-num}
%\bibliography{biblio3}
\bibliography{AAVFCP}

\end{document}